\documentclass{gtart}

\usepackage{psfrag, graphicx, subfigure, epsfig,color}

\usepackage{amsmath, amssymb, latexsym, euscript}

\usepackage{mathptmx, wrapfig}

\def\lst{\mathcal{S}}
\def\tri{\mathcal{T}}
\def\lay{\mathcal{L}}
\def\even{\mathfrak{e}}
\def\odd{\mathfrak{o}}

\def\bkR{{\rm I\kern-.17em R}}
\def\R{\bkR}
\def\bkZ{{\rm Z\kern-.28em Z}}
\def\Z{\bkZ}

\DeclareMathOperator{\e}{e}
\DeclareMathOperator{\T}{T}

\newcommand{\abs}[1]{\lvert#1\rvert}

\theoremstyle{plain}
\newtheorem{theorem}{Theorem}
\newtheorem*{theorem*}{Theorem}
\newtheorem{lemma}[theorem]{Lemma}
\newtheorem{proposition}[theorem]{Proposition}
\newtheorem{corollary}[theorem]{Corollary}

\theoremstyle{definition}
\newtheorem{definition}[theorem]{Definition}
\newtheorem*{definition*}{Definition}

\newtheorem{conjecture}[theorem]{Conjecture}

\theoremstyle{remark}
\newtheorem{remark}[theorem]{Remark}

\numberwithin{equation}{section}

\begin{document}

\title{Minimal triangulations for an infinite family of lens spaces}
\author{William Jaco, Hyam Rubinstein and Stephan Tillmann}

\begin{abstract} 
The notion of a layered triangulation of a lens space was defined by Jaco and Rubinstein in \cite{JR:LT}, and, unless the lens space is $L(3,1),$ a layered triangulation with the minimal number of tetrahedra was shown to be unique and termed its \emph{minimal layered triangulation}. This paper proves that for each $n\ge2,$ the minimal layered triangulation of the lens space $L(2n,1)$ is its unique minimal triangulation. More generally, the minimal triangulations (and hence the complexity) are determined for an infinite family of lens spaces containing the lens spaces $L(2n,1).$
\end{abstract}

\makeshorttitle


\section{Introduction}

Given a closed, irreducible 3--manifold, its complexity is the minimum number of tetrahedra in a (pseudo--simplicial) triangulation of the manifold. This number agrees with the complexity defined by Matveev~\cite{Mat1990} unless the manifold is $S^3,$ $\R P^3$ or $L(3,1).$ Matveev's complexity of these three manifolds is zero, but the respective minimum numbers of tetrahedra are one, two and two. It follows from the definition that the complexity is known for all closed manifolds which appear in certain computer generated censuses. Moreover, the question of determining the complexity of a given closed 3-manifold has an algorithmic solution which uses the solution to the \emph{Homeomorphism Problem for 3--Manifolds} and, in general, is impractical. It has been an open problem to determine the complexity for an infinite family of closed manifolds. Two sided asymptopic bounds are given for two families of closed hyperbolic manifolds by Matveev, Petronio and Vesnin \cite{MPV} using hyperbolic volume. Other two-sided bounds using homology groups have been obtained by Matveev \cite{Mat1990}, \cite{Mat2005}. This paper determines minimal triangulations (and hence the complexity) for an infinite family of closed, irreducible 3--manifolds, which includes the following three families:

\begin{theorem}\label{thm:main}
The minimal layered triangulation of the lens space $L(2n,1),$ $n\ge 2,$ is its unique minimal triangulation. The complexity of $L(2n,1)$ is therefore $2n-3.$
\end{theorem}

\begin{theorem}\label{thm:main2}
The minimal layered triangulation of $L( (s+2)(t+1)+1 , t+1),$ where $t>s>1,$ $s$ odd and $t$ even, is its unique minimal triangulation. The complexity of $L( (s+2)(t+1)+1 , t+1)$ is therefore $s+t.$
\end{theorem}

\begin{theorem}\label{thm:main3}
The minimal layered triangulation of $L( (s+1)(t+2)+1 , t+2),$ where $t>s>1,$ $s$ even and $t$ odd, is its unique minimal triangulation. The complexity of $L( (s+1)(t+2)+1 , t+2)$ is therefore $s+t.$
\end{theorem}

Layered triangulations of lens spaces and minimal layered triangulations of lens spaces are defined and studied in \cite{JR:LT}. A brief review of layered triangulations of the solid torus and lens spaces is given in Section~\ref{sec:triangulations}. The above results imply that the following conjecture holds for infinite families of lens spaces:

\begin{conjecture}\cite{JR:LT}\label{conj:layered lens is minimal}
A minimal layered triangulation of a lens space is a minimal triangulation; moreover, it is the unique minimal triangulation except for $S^3,$ $\R P^3$ and $L(3,1)$. 
\end{conjecture}

In particular, given $L(p,q),$ where $(p,q)=1,$ $p>q>0$ and $p>3,$ the conjectured complexity is $E(p,q)-3,$ where $E(p,q)$ is the number of steps needed in the Euclidean algorithm (viewed as a subtraction algorithm rather than a division algorithm) to transform the unordered tuple $(p,q)$ to the unordered tuple $(1,0).$ Equivalently, $E(p,q)$ is the sum, $\sum n_i,$ of the ``partial denominators" in the continued fractions expansion of $p/q:$
$$\frac{p}{q} = n_0 + \cfrac{1}{n_1+\cfrac{1}{n_2+\cfrac{1}{n_3+ \cfrac{1}{\ddots}}}}=[n_0, n_1, n_2, n_3, \dots].$$
There is an analogous conjecture in terms of special spines due to Matveev~\cite{Mat1990} giving the same conjectured complexity. The minimal known special spine of $L(p,q)$ is dual to the minimal layered triangulation. Anisov~\cite{Ani, Ani01} has recently found a geometric characterisation of the former as the cut locus of a generic point when $L(p,q),$ $q\neq 1,$ is imbued with its standard spherical structure. In the case $q=1,$ the cut locus is not a special spine, but perturbing either the metric or the cut locus slightly gives the minimal known special spine. 

The above conjectures were made in the 1990s, and this paper gives the first infinite family of examples which satisfy them. The new ingredients used to address the conjectures are the following.

First, we study minimal and 0--efficient triangulations of closed, orientable 3--manifolds via the edges of lowest degree. Any such triangulation contains at least two edges of degree at most five, and the neighbourhood of an edge of degree no more than five is described in Section \ref{sec:low degree edges}.

Second, assume that $M$ is a closed, orientable 3--manifold which admits a non-trivial homomorphism $\varphi\co \pi_1(M) \to \Z_2.$ Each (oriented) edge, $e,$ in a one-vertex triangulation of $M$ represents an element, $[e],$ of $\pi_1(M)$ and is termed \emph{$\varphi$--even} if $\varphi [e]=0$ and \emph{$\varphi$--odd} otherwise. We exhibit a relationship between the number of $\varphi$--even edges, the number of tetrahedra and the Euler characteristic of a canonical normal surface dual to $\varphi$ in Section \ref{sec:co-homology classes}. 

Third, we analyse in Section~\ref{sec:Intersections of maximal layered solid tori} how layered triangulations of solid tori which sit as subcomplexes of minimal and 0--efficient triangulations can intersect. 

As an application of these results, we prove the following in Section~\ref{sec:minimal of L(2n,1)}:
\begin{theorem}\label{thm:really main} 
A lens space with even fundamental group satisfies Conjecture \ref{conj:layered lens is minimal} if it has a minimal layered triangulation such that there are no more $\varphi$--odd edges than there are $\varphi$--even edges. 
\end{theorem}
It is then shown that this implies that a lens space never has more $\varphi$--even edges than $\varphi$--odd edges, and that a lens space satisfies this condition if it is contained in one of the families given in Theorems \ref{thm:main}, \ref{thm:main2} or \ref{thm:main3}. Moreover, a complete description of the set of all lens spaces to which Theorem \ref{thm:really main} applies is given in Figure \ref{fig:L-graph}.

The first author is partially supported by NSF Grant DMS-0505609 and the Grayce B. Kerr Foundation. The second and third authors are partially supported under the Australian Research Council's Discovery funding scheme (project number DP0664276).


\section{Layered triangulations}
\label{sec:triangulations}

This section establishes notation and basic definitions, and describes a family of layered triangulations of the solid torus, denoted $\lst_k = \{1, k+1, k+2\},$ as well as the minimal layered triangulation of the lens space $L(k+3,1),$ $k \ge 1.$ A consequence of Theorem~\ref{thm:main} concerning minimal triangulations of the solid torus subject to a specified triangulation of the boundary is also given (see Corollary \ref{cor:main thm}).


\subsection{Triangulations}

The notation of \cite{JR, JR:LT} will be used in this paper. A triangulation, $\tri,$ of a 3--manifold $M$ consists of a union of pairwise disjoint 3--simplices, $\widetilde{\Delta},$ a set of face pairings, $\Phi,$ and a natural quotient map $p\co \widetilde{\Delta} \to \widetilde{\Delta} / \Phi = M.$ Since the quotient map is injective on the interior of each 3--simplex, we will refer to the image of a 3--simplex in $M$ as a \emph{tetrahedron} and to its faces, edges and vertices with respect to the pre-image. Similarly for images of 2-- and 1--simplices, which will be referred to as \emph{faces} and \emph{edges} in $M.$ For edge $e,$ the number of pairwise distinct 1--simplices in $p^{-1}(e)$ is termed its \emph{degree}, denoted $d(e).$ If an edge is contained in $\partial M,$ then it is termed a \emph{boundary edge}; otherwise it is an \emph{interior edge}.


\subsection{Minimal and 0--efficient triangulations}

A triangulation of the closed, orientable, connected 3--manifold $M$ is termed \emph{minimal} if $M$ has no other triangulation with fewer tetrahedra. A triangulation of $M$ is termed \emph{0--efficient} if the only embedded, normal 2--spheres are vertex linking. It is shown by the first two authors in \cite{JR} that (1) the existence of a 0--efficient triangulation implies that $M$ is irreducible and $M\neq \R P^3,$ (2) a minimal triangulation is 0--efficient unless $M=\R P^3$ or $L(3,1),$ and (3) a 0--efficient triangulation has a single vertex unless $M=S^3$ and the triangulation has precisely two vertices. These facts will be used implicitly.


\subsection{Layered triangulations of the solid torus and lens spaces}
\label{sec:Layered triangulations of the solid torus and lens spaces}

Layering along a boundary edge is defined in \cite{JR:LT} and illustrated in Figure \ref{fig:layering}. Namely, suppose $M$ is a 3--manifold, $\tri_\partial$ is a triangulation of $\partial M$, and $e$ is an edge in $\tri_\partial$ which is incident to two distinct faces. We say the 3--simplex $\sigma$ is \emph{layered along the (boundary) edge} $e$ if two faces of $\sigma$ are paired, ``without a twist," with the two faces of $\tri_\partial$ incident with $e$. The resulting 3--manifold is homeomorphic with $M$. If $\tri_\partial$ is the restriction of a triangulation $\tri$ of $M$ to $\partial M$, then we get a new triangulation of $M$ and denote it $\tri\cup_e \sigma$.

\begin{figure}[t]
\psfrag{M}{{\small $M$}}
\psfrag{S}{{\small $\Delta$}}
\psfrag{e}{{\small $e_1$}}
\psfrag{f}{{\small $e_2$}}
\psfrag{g}{{\small $e_3$}}
\begin{center}
    \subfigure[Layering on a boundary edge]{\label{fig:layering}
      \includegraphics[height=3.6cm]{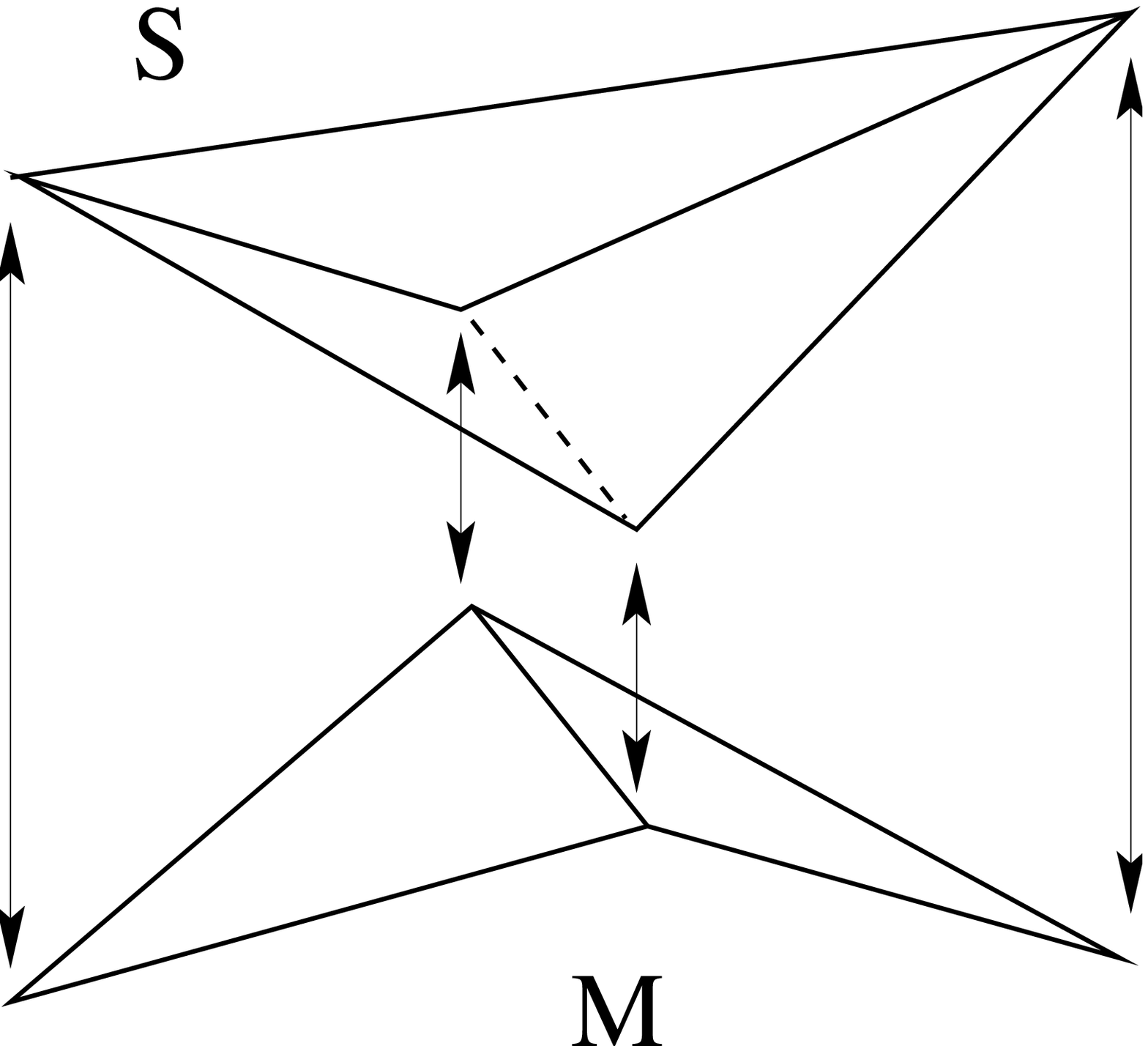}
    } 
    \qquad\qquad
    \subfigure[The solid torus $\lst_1$]{\label{fig:solid_torus}
      \includegraphics[height=3.6cm]{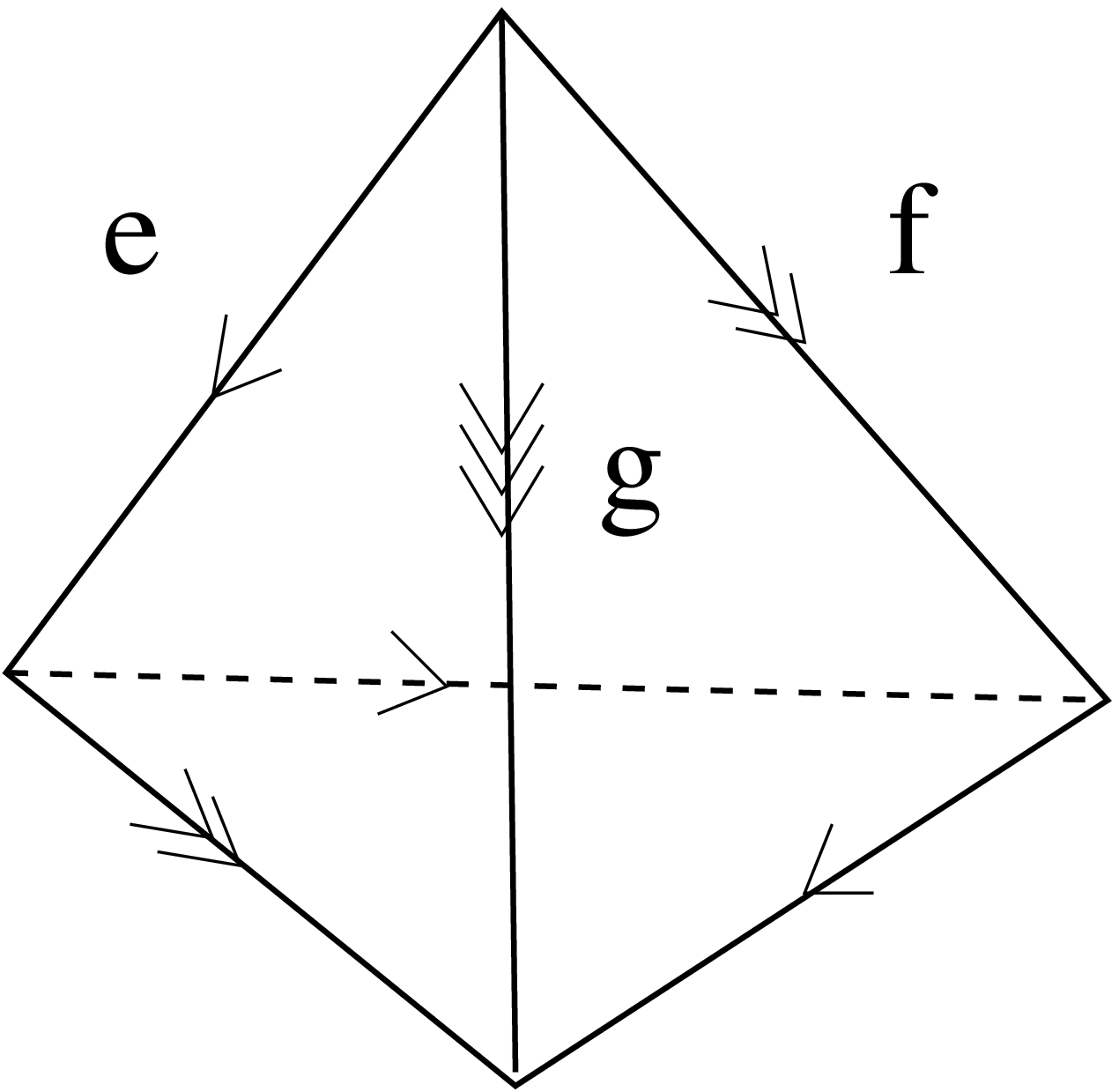}
    } 
\end{center}
    \caption{Layered triangulation of the solid torus}
     \label{fig:regular exchange is associative}
\end{figure}

Starting point for a layered triangulation of a solid torus is the one-tetrahedron triangulation of the solid torus shown in Figure \ref{fig:solid_torus}, where the two back faces are identified as indicated by the edge labels. One can then layer on any of the three boundary edges (see \cite{JR:LT}), giving a layered triangulation of the solid torus with two tetrahedra. Inductively, a layered triangulation of a solid torus with $k$ tetrahedra is obtained by layering on a boundary edge of a layered triangulation with $k-1$ tetrahedra. A layered triangulation of a solid torus with $k$ tetrahedra has one vertex, $2k+1$ faces, and $k+2$ edges; the vertex, three edges, and two faces are contained in the boundary. There is a special edge in the boundary having degree one; it is called the \emph{univalent edge}.

Given a solid torus with a one-vertex triangulation on its boundary, it follows from an Euler characteristic argument that this triangulation has two faces meeting along three edges. There is a unique set of unordered nonnegative integers, $\{p,q,p+q\}$, associated with the triangulation on the boundary and determined from the geometric intersection of the meridional slope of the solid torus with the three edges of the triangulation. Two such triangulations with numbers $\{p,q,p+q\}$ and $\{p',q',p'+q'\}$ can be carried to each other via a homeomorphism of the solid torus if and only if the two sets of numbers are identical. There are
three possible ways to identify the two faces in the boundary, each giving a lens space. These identifications can be thought of as ``folding along an edge in the boundary", which is also referred to as ``closing the book" with an edge as the binding. Folding along the edge $p$ gives the lens space $L(2q+p,q)$; along the edge $q$ the lens space $L(2p+q,p)$; and along the edge $p+q$ the lens space $L(\abs{p-q},p)$. If the solid torus is triangulated, then after identification one obtains a triangulation of the lens space. If the triangulation of the solid torus is a layered triangulation of the solid torus, then the induced triangulation of the lens space is termed a \emph{layered triangulation} of the lens space. A layered triangulation of a lens space having the the minimal number of tetrahedra is called a \emph{minimal layered triangulation} of the lens space. It is shown in \cite{JR:LT}, that ---except for the lens space $L(3,1)$--- a lens space has a unique minimal layered triangulation. It is conjectured, see Conjecture~\ref{conj:layered lens is minimal}, that a minimal layered triangulation is minimal.

There is a companion conjecture for triangulations of the solid torus. If $\tri_{\partial}$ is a triangulation on the boundary of a compact $3$--manifold $M$, then a triangulation $\tri$ of $M$ having all its vertices in ${\partial} M$ and agreeing with $\tri_{\partial}$ on $\partial M$ is called an \emph{extension of $\tri_\partial$} to $M.$ A one-vertex triangulation on the boundary of a solid torus can be extended to a layered triangulation of the solid torus. There are infinitely many ways to extend the triangulation on the boundary; however, it is shown in \cite{JR:LT} that there is a unique layered extension having the minimal number of tetrahedra; termed the minimal layered extension. The companion conjecture to that given above is that the minimal layered extension is the minimal extension; a first infinite family satisfying this companion conjecture is given in Corollary~\ref{cor:main thm}.


\subsection{The minimal layered extension of $\mathbf{\{k+2,k+1,1\}}$}
\label{sec:minimal layered extension}

We now describe the combinatorics of a layered triangulation of the solid torus extending the $\{k+2,k+1,1\}$ triangulation on the boundary of the solid torus. It is shown below that this is the minimal layered extension. This family of layered triangulations can be characterized by the fact that there is a unique interior edge of degree three and all other interior edges have degree four. Denote by $\lst_1$ the triangulation of the solid torus with one tetrahedron, $\Delta_1,$ with edges labelled and oriented as shown in Figure \ref{fig:solid_torus}. A 3--simplex, $\Delta_2,$ is layered on edge $e_2;$ there is a unique edge, labelled $e_4,$ which is not identified with any of $e_1, e_2, e_3.$ The edge $e_4$ is oriented such that its incident faces give the relation $e_4 = e_1 + e_3.$ Denote the resulting triangulated solid torus by $\lst_2.$ Inductively, $\Delta_k$ is layered on $e_k;$ the new edge is labelled $e_{k+2}$ with relation $e_{k+2} = e_1 + e_{k+1}.$ The resulting triangulated solid torus is denoted $\lst_k.$

Identify the fundamental group of the solid torus with $\Z.$ Analysing the meridian disc in the triangulation consisting of just $\Delta_1,$ and then taking into account the relations given from faces, one obtains:
$$[e_k] = k \in \Z,$$
with the given orientation conventions for edges. The edges are accordingly termed \emph{odd} or \emph{even}. When $k=2,$ we have $d(e_1)=5, d(e_2)=d(e_3)=3$ and $d(e_4)=1.$ After layering $k$ tetrahedra, $k \ge 3,$ the degrees of edges are as follows:
\begin{table}[h]
\begin{center}
\begin{tabular}{l | l | l | l | l | l}
edge		& $e_1$	&$e_2$	&$e_3$ to $e_k$	& $e_{k+1}$	& $e_{k+2}$ \\
\hline
degree	& $2k+1$	&3		&4				& 3			&1\\
\end{tabular}
\end{center}
\end{table}


\subsection{The minimal layered triangulation of the lens space $\mathbf{L(k+3, 1)}$}

In the above construction, at stage $k,$ one can choose not to add $\Delta_{k+1},$ but instead to fold along the edge $e_{k+1}$, identifying the two faces and identifying the edge $e_{k+2}$ with $-e_1$. As noted above, one obtains a layered triangulation of the lens space $L(k+3,1)$ with fundamental group $\Z_{k+3}.$ This is in fact the minimal layered triangulation of $L(k+3,1)$ as can be deduced from the continued fraction expansion. This also implies that the above extension of $\{k+2,k+1,1\}$ is the minimal layered extension. The previous discussion shows that there are $k+1$ edges, and their degrees are as follows if $k\ge3$:
\begin{table}[h]
\begin{center}
\begin{tabular}{l | l | l | l | l}
edge		& $e_1$	&$e_2$	&$e_3$ to $e_k$	& $e_{k+1}$ \\
\hline
degree	& $2k+2$	&3		&4				& 3\\
\end{tabular}
\end{center}
\end{table}

When $k=2,$ there are two edges of degree three and one of degree six, when $k=1,$ there is one of degree four and one of degree two.

This triangulation of $L(k+3, 1)$ is denoted $\lay=\lay_k.$ If $k$ is odd, then having $e_{k+2}$ identified with $-e_1$, each of the $k+1$ edges in $\lay_k$ can be termed \emph{odd} or \emph{even} without ambiguity and with consistency with their designation as ``odd" or ``even" in $\lst_k.$ Letting $n=\frac{k+3}{2},$ we have $L(k+3, 1)=L(2n,1),$ where $n\ge 2.$ In the sequel, we write $L(k+3, 1)=L(2n,1)$ to indicate that $n\ge 2$ and $k$ is an odd, positive integer.

Theorem~\ref{thm:main} states that $\lay_k$ is the unique minimal triangulation of $L(k+3, 1),$ $k\ge 1$ odd. This has the following consequence:

\begin{corollary}\label{cor:main thm}
The minimal layered extension of the triangulations $\{k+2,k+1,1\},$ $k\ge 1,$ on the boundary of the solid torus is the unique minimal extension.
\end{corollary}

\begin{proof}
Let $k \ge 1$ be odd. Suppose there is a minimal extension, $E,$ of $\{k+2,k+1,1\}$ which is not the minimal layered extension. Folding along the edge $k+1$ gives a triangulation of $L(k+3, 1)$ with $k$ or fewer tetrahedra. Since $\lay_k$ is the unique minimal triangulation of $L(k+3, 1),$ it follows that $E$ contains precisely $k$ tetrahedra and that it is obtained by splitting $\lay_k$ open along a face. Now in $\lay_k$ there are precisely two faces along which one can split open to obtain the minimal layered extension of $\{k+2, k+1, 1\},$      
hence not $E.$ Splitting open along any other face gives boundary a pinched 2--sphere, hence again not $E.$ We thus arrive at a contradiction.

If $k$ is even, assume that there is an extension of $\{k+2,k+1,1\}$ with $k$ or fewer tetrahedra which is not the minimal layered extension. One may layer another tetrahedron on this to get an extension of $\{(k+1)+2,(k+1)+1,1\}$ with $k+1$ or fewer tetrahedra and which  is not the minimal layered extension. This contradicts the above.
\end{proof}


\section{Edges of low degree in minimal triangulations}
\label{sec:low degree edges}

Given any 1--vertex triangulation of a closed 3--manifold, let $E$ denote the number of edges, $E_i$ denote the number of edges of degree $i,$ and $T$ denote the number of tetrahedra. Then $E = T+1,$ and hence $6 = \sum (6-i)E_i.$ It follows that there are at least two edges of degree at most five. The new contribution in this section is an analysis of edges of degree four and five in minimal and 0--efficient triangulations. For edges of degree at most three, this has been done previously by the first two authors. For convenience, Proposition~6.3 of \cite{JR} is stated (with minor corrections) as Propositions~\ref{pro:degree one and two edges} and \ref{pro:degree three edges} in Subsection~\ref{subsec:lowest degree}. A new proof is given for the case of degree three edges which generalises to degrees four and five. 

As a simple application of below characterisation of low degree edges, the reader may find pleasure in determining all closed, orientable, connected and irreducible 3--manifolds having a minimal and 0--efficient triangulation with at most three tetrahedra.


\subsection{Lowest degree edges}
\label{subsec:lowest degree}

\begin{lemma}\label{lem:dunce gives s3}
Suppose the closed, orientable, connected 3--manifold $M$ has a 0--efficient triangulation. If there is a face in $M$ which is a dunce hat, then $M=S^3.$
\end{lemma}

\begin{proof}
The face which is a dunce hat is bounded by a single edge. Since a dunce hat is contractible, the edge bounds an immersed disc in $M.$ Dehn's lemma (as stated in \cite{J}, I.6) implies that the edge bounds an embedded disc. Now \cite{JR}, Proposition~5.3, implies that $M=S^3.$
\end{proof}

\begin{proposition}[Edges of degree one or two]\cite{JR}\label{pro:degree one and two edges}
A minimal and 0--efficient triangulation $\tri$ of the closed, orientable, connected and irreducible 3--manifold $M$ has
\begin{enumerate}
\item[(1)] no edge of order one unless $M=S^3,$ and
\item[(2)] no edge of order two unless $M=L(3,1)$ or $L(4,1).$
\end{enumerate}
\end{proposition}                                                                    

\begin{proposition}[Edges of degree three]\cite{JR}\label{pro:degree three edges}
A minimal and 0--efficient triangulation $\tri$ of the closed, orientable, connected and irreducible 3--manifold $M$ has no edge of order three unless either
\begin{enumerate}
\item[(3a)] $\tri$ contains a single tetrahedron and $M=L(5,2);$ or
\item[(3b)] $\tri$ contains two tetrahedra and $M=L(5,1)$ or $L(7,2);$ or
\item[(3c)] $\tri$ contains, as an embedded subcomplex, the two tetrahedron, layered triangulation $\lst_2 =\{4,3,1\}$ of the solid torus. Moreover, $\tri$ contains at least three tetrahedra and every edge of degree three is contained in such a subcomplex.
\end{enumerate}
Moreover, the triangulations in (3b) are obtained by identifying the boundary faces of $\lst_2$ appropriately.
\end{proposition}
\begin{proof}
Assume that $e$ is an edge of degree three in $\tri.$ Let $\widetilde{\Delta}_e \subset \widetilde{\Delta}$ be the set of all 3--simplices containing a pre-image of $e.$ Then $1 \le |\widetilde{\Delta}_e| \le 3.$

If $|\widetilde{\Delta}_e| = 1,$ then the preimage of a small loop around $e$ must meet all four faces of the single tetrahedron in $\widetilde{\Delta}_e.$ It follows that $\widetilde{\Delta}=\widetilde{\Delta}_e.$ Analysing the possibilities gives $M=L(5,2)$ and $\tri$ is the a one-tetrahedron triangulation thereof. We are thus in case (3a).

If $|\widetilde{\Delta}_e| = 2,$ then one of the tetrahedra in $\widetilde{\Delta}_e,$ denoted $\widetilde{\sigma}_1,$ contains exactly one pre-image of $e,$ and the other, denoted $\widetilde{\sigma}_2,$ contains exactly two. The faces $f^1_0, f^1_1$ of $\widetilde{\sigma}_1$ meeting in the pre-image of $e$ have to be identified with faces $f^2_0, f^2_1$ respectively of $\widetilde{\sigma}_2.$ First assume that $f^2_0$ and $f^2_1$ meet in a pre-image of $e.$ Since $d(e) \neq 2,$ one of $f^2_0, f^2_1$ must contain another pre-image of $e.$ But this implies that $d(e) > 3$ since $M$ is closed and $|\widetilde{\Delta}_e| = 2.$ It follows that $f^2_0$ and $f^2_1$ cannot meet in a pre-image of $e.$ Similarly, if the pre-images of $e$ in the two faces meet in a vertex, then we get a contradiction to $d(e) = 3.$ Hence the pre-images form a pair of opposite edges and the remaining faces, $f^2_2, f^2_3,$ of $\widetilde{\sigma}_2$ must be identified since $|\widetilde{\Delta}_e| = 2.$ The face pairings $f^1_0\leftrightarrow f^2_0,$ $f^1_1\leftrightarrow f^2_1,$ $f^2_2 \leftrightarrow f^2_3$ are uniquely determined by this information and the fact that $M$ is orientable. The resulting identification space, $X,$ is equivalent to the layered triangulation $\lst_2$ of the solid torus. The projection $p\co \widetilde{\Delta} \to M$ gives rise to a natural map $p'\co X \to M$ such that the restriction of $p$ to $\widetilde{\sigma}_1 \cup \widetilde{\sigma}_2$ factors through $p'.$

If $p'$ is an embedding, then we are in case (3c). Hence assume that there are further identifications between $\widetilde{\sigma}_1$ and $\widetilde{\sigma}_2.$ 

If there is a pairing between the remaining faces of $\widetilde{\sigma}_1,$ then $\widetilde{\Delta}=\widetilde{\Delta}_e.$ Analysing the three possibilities (recalling that $M$ is orientable) gives either of the following cases:
\begin{enumerate}
\item[(1)] Folding along the univalent edge gives $\R P^3.$ However, the triangulation is not 0--efficient.
\item[(2)] Folding along the edge of degree five gives $L(7,2).$
\item[(3)] Folding along the edge of degree three gives $L(5,1).$
\end{enumerate}
These are the possibilities stated in case (3b).

Now assume there is no further identification of faces but of edges between $\widetilde{\sigma}_1$ and $\widetilde{\sigma}_2.$ Note that there are three distinct edges in the boundary of $X.$ If precisely two of them are identified, then either of the following cases applies:
\begin{enumerate}
\item[(4)] The triangulation $\tri$ contains a face which is a cone. Then \cite{JR}, Corollary 5.4, implies that $M=S^3.$ But a minimal triangulation of $S^3$ contains a single tetrahedron, contradicting $|\widetilde{\Delta}| > 2.$
\item[(5)] The boundary, $F,$ of a regular neighbourhood of $p(\widetilde{\sigma}_1 \cup \widetilde{\sigma}_2)$ in $M$ is an embedded 2--sphere and a barrier surface (see \cite{JR}). Hence $F$ shrinks to a stable surface in the complement of $p(\widetilde{\sigma}_1 \cup \widetilde{\sigma}_2).$ Since every normal 2--sphere in a 0--efficient triangulation is vertex linking and the complement of $p(\widetilde{\sigma}_1 \cup \widetilde{\sigma}_2)$ cannot contain a vertex linking surface, $F$ bounds a 3--ball in the complement of $p(\widetilde{\sigma}_1 \cup \widetilde{\sigma}_2).$
There is a homotopy of $M$ taking this ball to a disc which extends to a homotopy identifying the two free faces of $p(\widetilde{\sigma}_1 \cup \widetilde{\sigma}_2),$ hence giving rise to a triangulation of $M$ with fewer tetrahedra than $\tri.$ This contradicts minimality of $\tri.$
\end{enumerate}
Now assume that all three edges in the boundary of $X$ are identified. Then either a face is a dunce hat and Lemma \ref{lem:dunce gives s3} yields a contradiction as in (4), or the argument in (5) applies.

If $|\widetilde{\Delta}_e| = 3,$ then $e$ is contained in three distinct tetrahedra and a $3 \rightarrow 2$ Pachner move can be applied. This reduces the number of tetrahedra, contradicting the assumption that $\tri$ is minimal. 
\end{proof}


\subsection{Edges of degree equal to four}

The analysis of degree four and five edges is done more coarsely in order to reduce the number of cases to be considered. Some of the notions in the preceding proof are first formalised. Given an edge, $\overline{e},$ of degree $n$ in a minimal and 0--efficient triangulation $\tri$ of the closed, orientable, connected and irreducible 3--manifold $M,$ assume that it is contained in precisely $k$ pairwise distinct tetrahedra, $\sigma_1,...,\sigma_k.$ Then $1\le k \le n$ and there is a triangulated complex, $X=X_{n;k},$ having a triangulation with $k$ tetrahedra containing an interior edge of degree $n,$ denoted $e,$ and a map $p_e\co X \to M$ taking $e \to \overline{e}$ with the property that 
$p|_{\{\widetilde{\sigma}_1, ..., \widetilde{\sigma}_k\}}\co \widetilde{\Delta} \to M$ factors through $p_e.$ A complex $X$ as above is \emph{maximal} if there is no other such complex $X'$ with the property that 
$$p\co \{\widetilde{\sigma}_1, ..., \widetilde{\sigma}_k\} \to X \to M\quad\text{factors as}\quad
p\co \{\widetilde{\sigma}_1, ..., \widetilde{\sigma}_k\} \to X' \to X \to M.$$
If $X_{n;k}$ is maximal, then $(X_{n;k}, e)$ is said to be a \emph{model for $\overline{e}.$}

{
\begin{figure}[t]
\psfrag{M}{{\small $\lst_1$}}
\psfrag{N}{{\small $\lst_1$}}
\psfrag{e}{{\small $e$}}
\begin{center}
    \subfigure[$X^0_{4;1} \cong S^3$]{\label{fig:X10}
      \includegraphics[height=2.1cm]{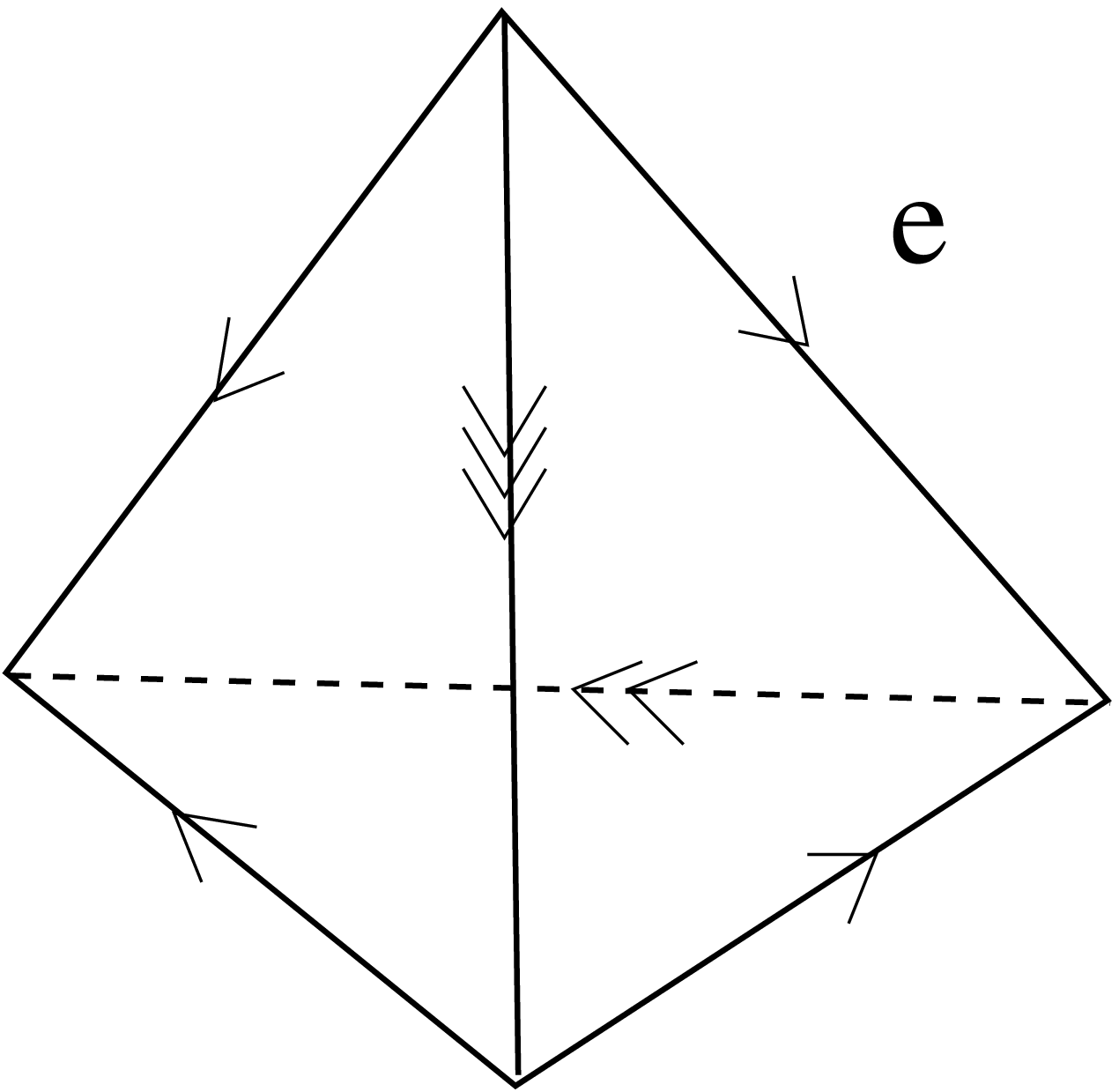}
    } 
    \qquad
    \subfigure[$X^1_{4;1} \cong L(4,1)$]{\label{fig:X11}
      \includegraphics[height=2.1cm]{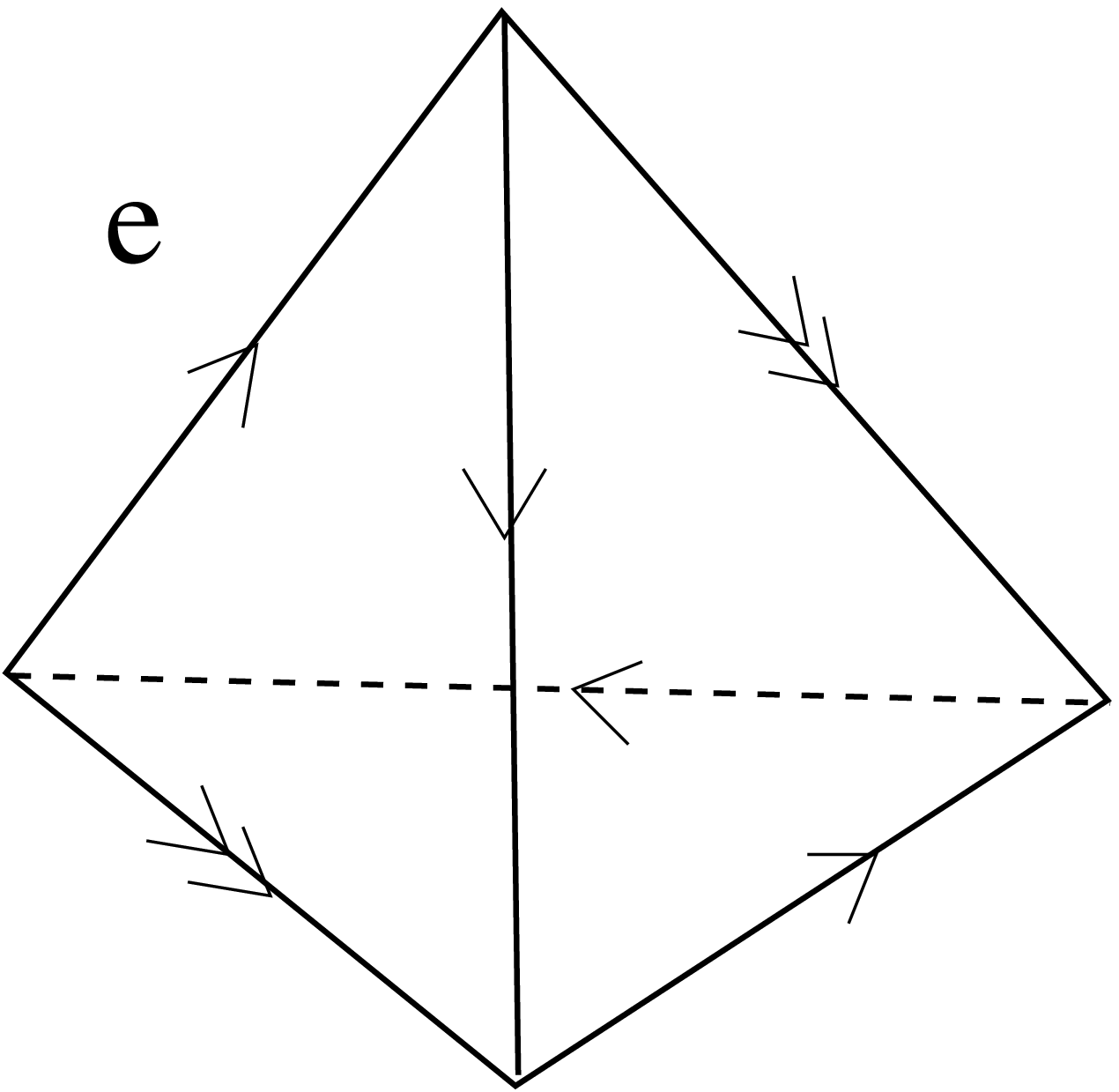}
    } 
   \qquad
           \subfigure[$X_{4;2}^0\cong L(8,3)$]{\label{fig:X21}
      \includegraphics[height=2.1cm]{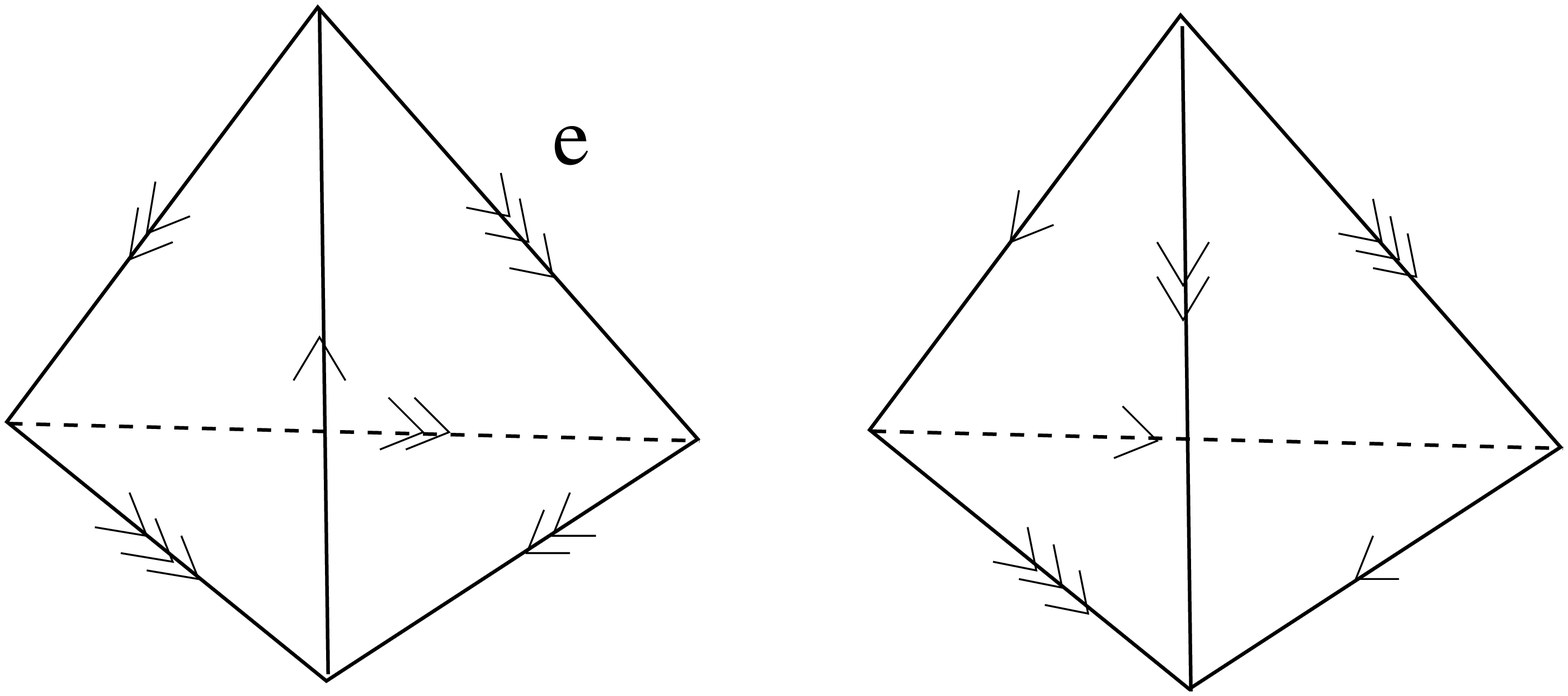}
    } 
   \\
           \subfigure[$X_{4;2}^1\cong S^3/Q_8$]{\label{fig:X22}
      \includegraphics[height=2.2cm]{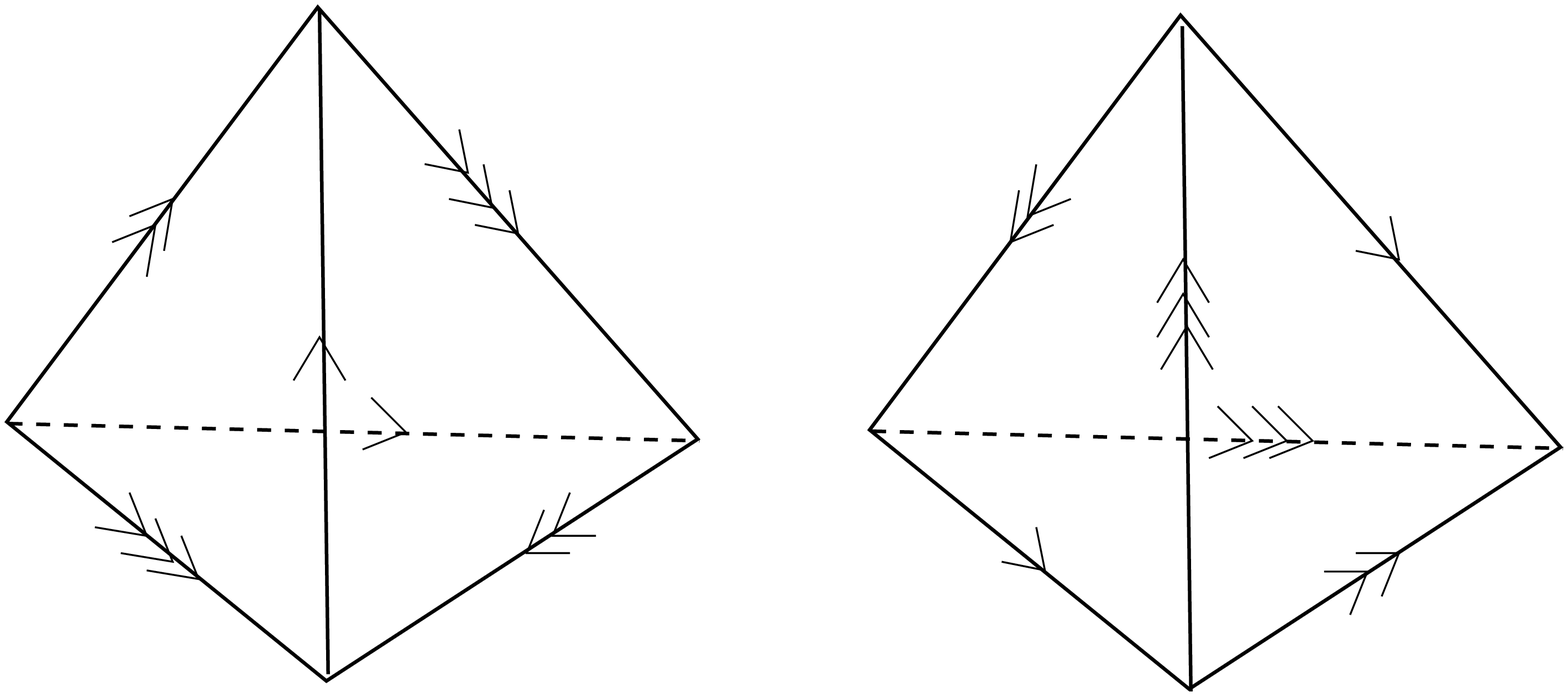}
    } 
    \qquad
           \subfigure[$X_{4;2}^2\cong$ solid torus]{\label{fig:X2}
      \includegraphics[height=2.3cm]{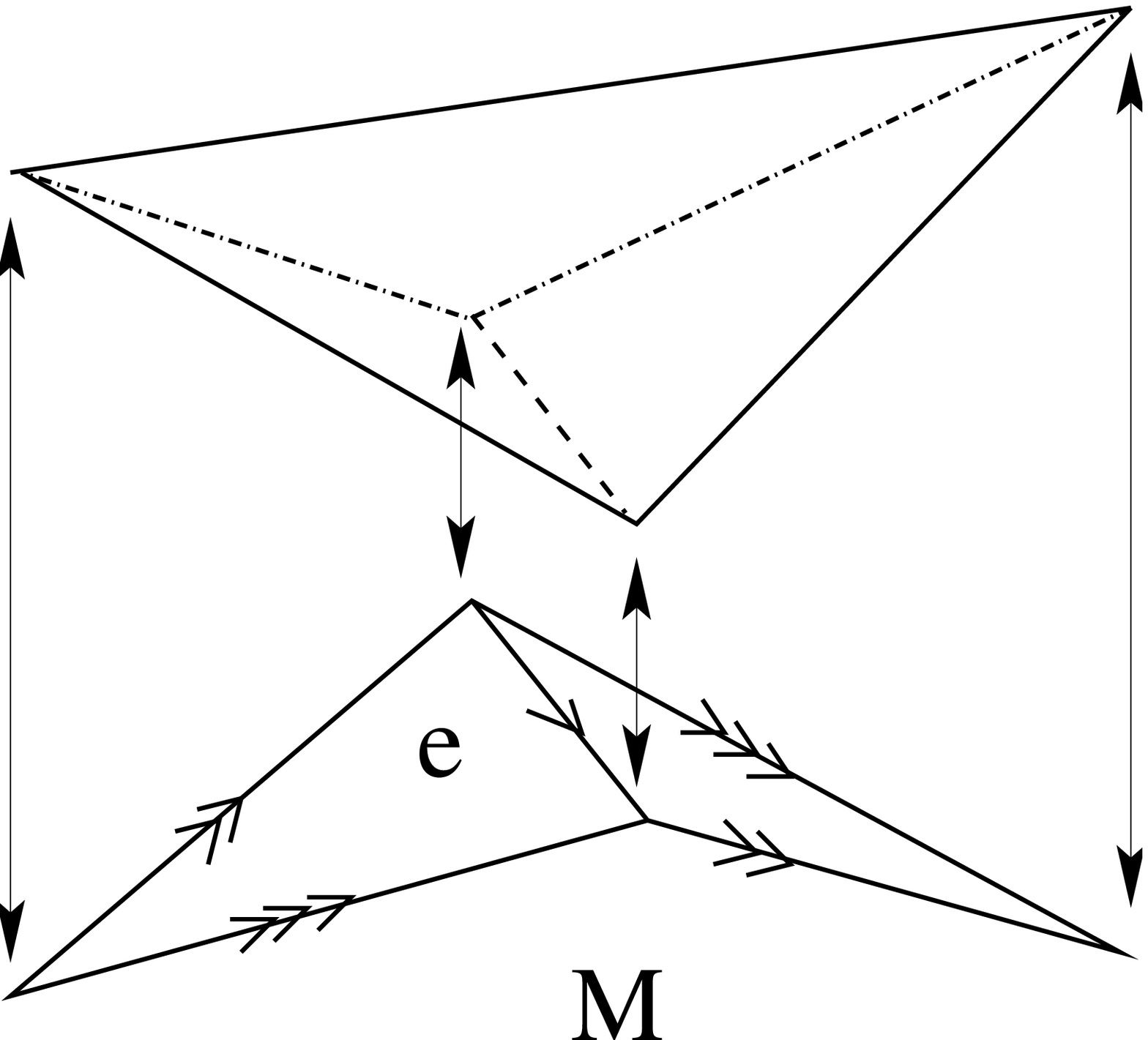}
    } 
    \qquad
    \subfigure[$X^0_{4;3}\cong$ solid torus]{\label{fig:X30}
      \includegraphics[height=2.3cm]{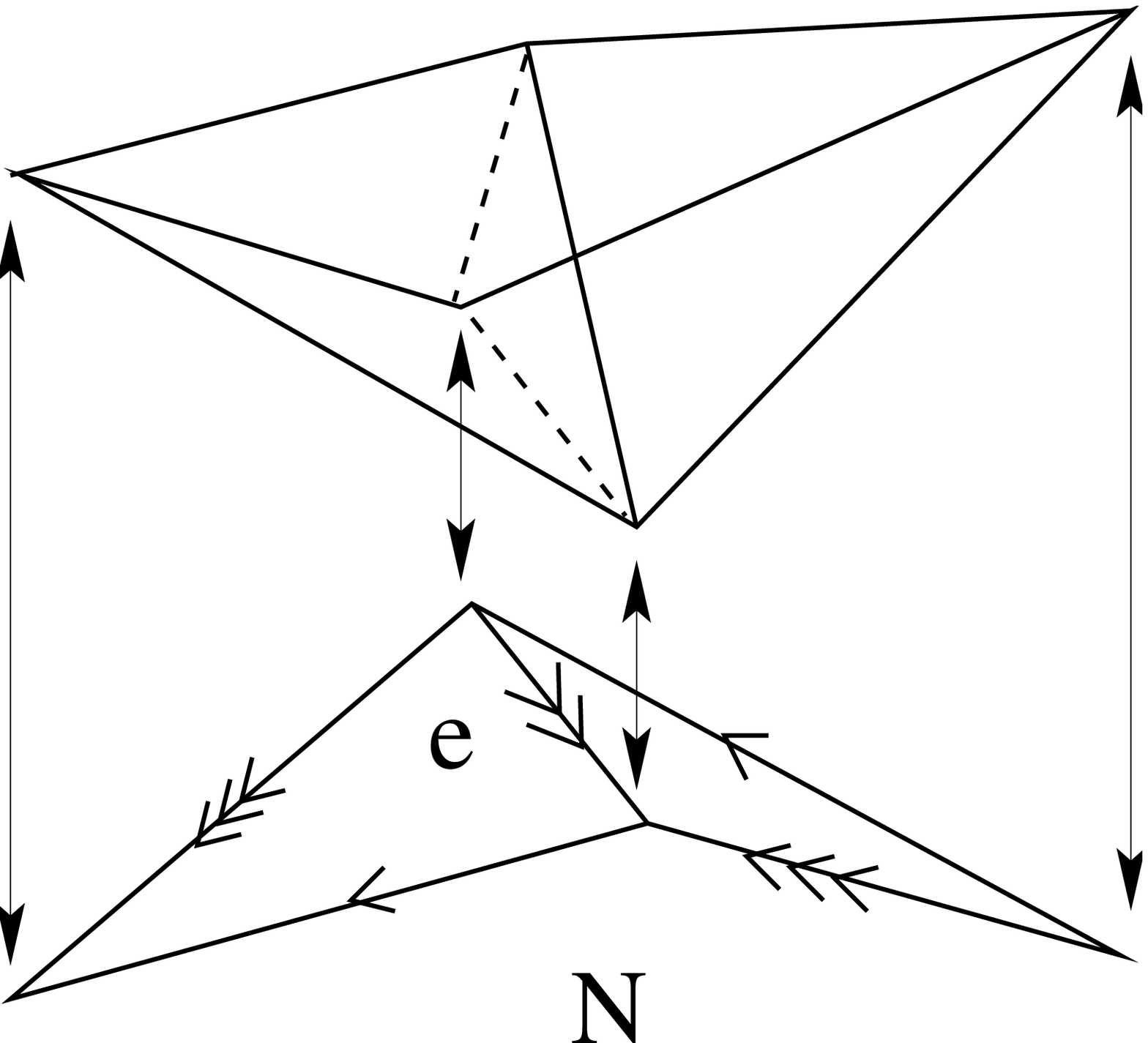}
    } 
    \\
    \subfigure[$X^1_{4;3}\cong$ solid torus]{\label{fig:X31}
      \includegraphics[height=2.2cm]{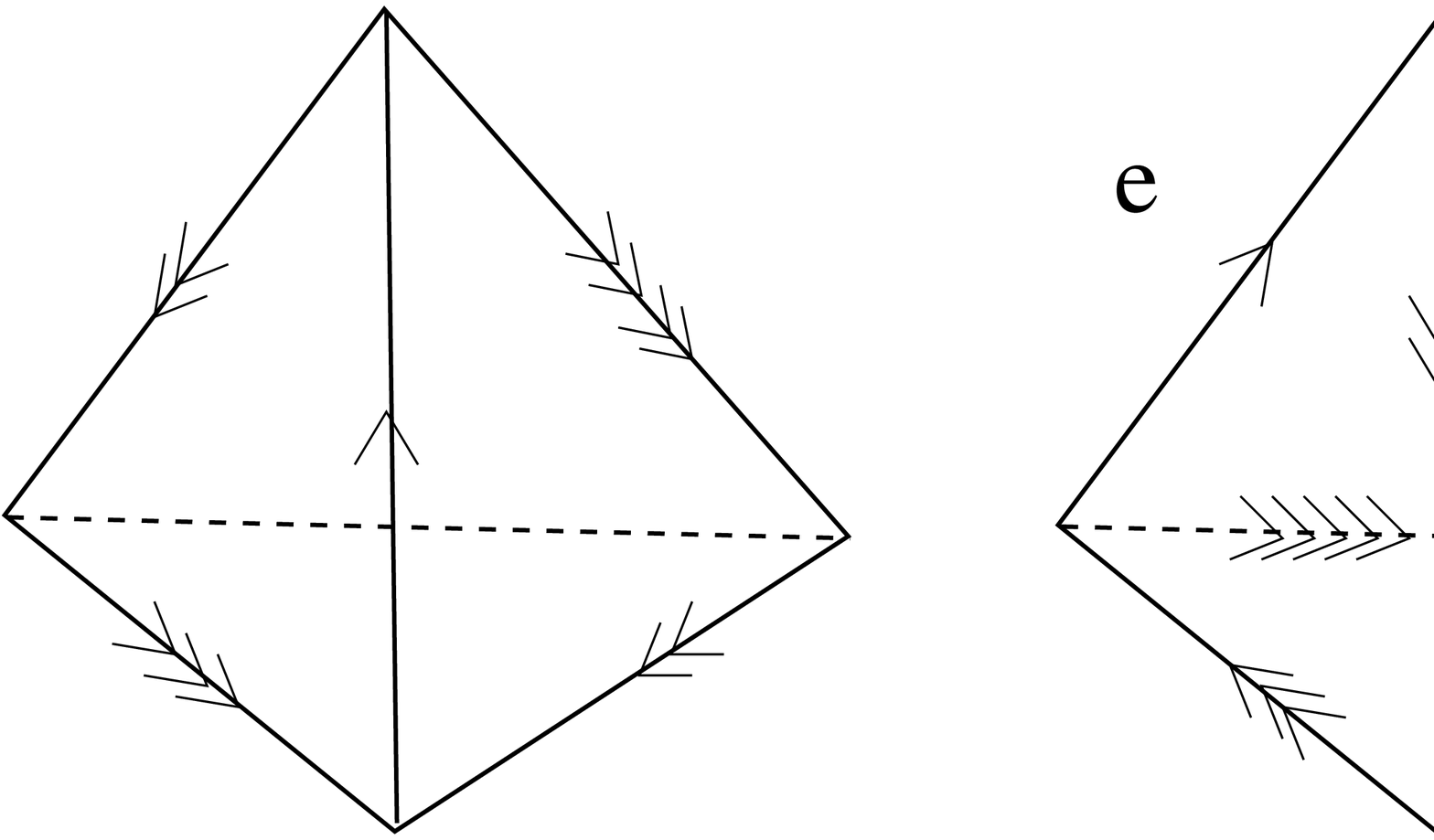}
    } 
     \qquad
         \subfigure[$X_{4;4}\cong$ 3--ball]{\label{fig:X4}
      \includegraphics[height=2.2cm, width=3cm]{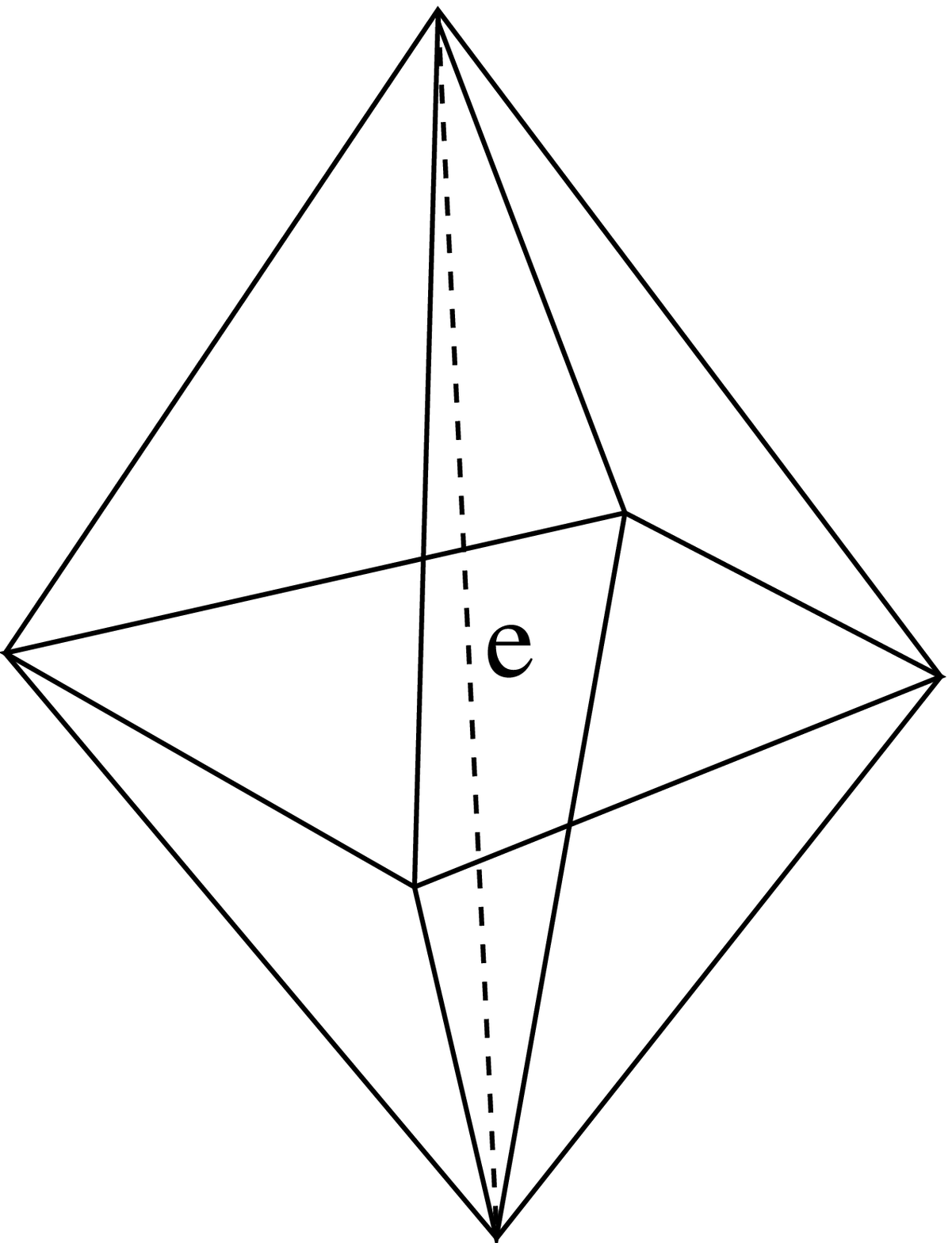}
    } 
\end{center}
    \caption{Degree four edges in minimal and 0--efficient triangulations}
     \label{fig:degree four edges}
\end{figure}
}

\begin{proposition}[Edges of degree four]\label{pro:degree four edges}
Assume the minimal and 0--efficient triangulation $\tri$ of the closed, orientable, connected and irreducible 3--manifold $M$ has an edge of degree four, denoted $\overline{e}.$ Then the model, $(X_{4;k}, e),$ for $\overline{e}$ is one of the following:
\begin{enumerate}
\item[(4a)] If $k=1,$ then $X_{4;1}=X_{4;1}^0$ or $X_{4;1}^1,$ where $X_{4;1}^0$ and $X_{4;1}^1$ are one-tetrahedron triangulations of $S^3$ and $L(4,1)$ respectively that contain a (necessarily unique) edge of degree four, $e.$ In particular, $M=S^3$ or $L(4,1).$
\item[(4b)] If $k=2,$ then $X_{4;2}=X_{4;2}^0,$ $X_{4;2}^1$ or $X_{4;2}^2,$ where $X_{4;2}^0$ is a two-tetrahedron triangulation of $L(8,3)$ with $e$ the unique edge which has degree two with respect to each tetrahedron; $X_{4;2}^1$ is a two-tetrahedron triangulation of $S^3/Q_8$ with $e$ either of its three degree four edges; and $X_{4;2}^2 = \{5,2,3\}$ is the triangulated solid torus $\lst_1$ with another tetrahedron layered on the boundary edge of degree three with $e$ the unique edge of degree four;
\item[(4c)] If $k=3,$ then $X_{4;3}=X_{4;3}^0$ or $X_{4;3}^1,$ where $X_{4;3}^0$ is $\lst_1$ with two tetrahedra attached to its boundary faces such that the degree 2 boundary edge becomes an interior edge of degree four, $e;$ and $X_{4;3}^1$ is a solid torus obtained by first identifying a pair of opposite edges of a tetrahedron, and then layering two tetrahedra on the resulting edge making it an interior edge of degree four, $e.$
\item[(4d)] If $k=4,$ then $X_{4;4}$ is an octahedron triangulated with four tetrahedra with $e$ their common intersection.
\end{enumerate}
The maximal complexes are shown in Figure \ref{fig:degree four edges}: To obtain $X_{4;1}^0$ and $X_{4;1}^1,$ identify the two front faces and identify the two back faces as indicated by the arrows; to obtain $X_{4;2}^0$ identify the two front faces of the first 3--simplex with the front faces of the second, and then the back faces of each tetrahedron; for $X_{4;2}^1$ identify the two front faces of the first 3--simplex with the front faces of the second, and likewise with the back faces; to obtain $X_{4;3}^1,$ identify the front faces of the first 3--simplex with the front faces of the second, and the back faces of the second with the front faces of the third.
\end{proposition}

\begin{proof}
We only have to show that the stated list of maximal complexes is correct. The main arguments are as in the proof for Proposition \ref{pro:degree three edges}, and we will therefore not give all details. As above, assume that $e$ is an edge of degree four in $M,$ and let $\widetilde{\Delta}_e \subset \widetilde{\Delta}$ be the set of all 3--simplices containing a pre-image of $e.$ Then $1 \le |\widetilde{\Delta}_e| \le 4.$

If $|\widetilde{\Delta}_e| = 1,$ then the preimage of a small loop around $e$ must meet all four faces of the single tetrahedron in $\widetilde{\Delta}_e.$ It follows that $\widetilde{\Delta}=\widetilde{\Delta}_e.$ Analysing the possibilities gives $M = S^3$ or $L(4,1)$ and $\tri$ is a one-tetrahedron triangulation containing a (necessarily unique) edge of degree four. We are thus in case (4a).

If $|\widetilde{\Delta}_e| = 2,$ denote the tetrahedra in $\widetilde{\Delta}_e$ by $\widetilde{\sigma}_1$ and $\widetilde{\sigma}_2.$ Without loss of generality, $\widetilde{\sigma}_1$ contains precisely one or two pre-images of $e.$ 

In the first case, analysing the possible face pairings yields the one tetrahedron solid torus with another tetrahedron layered on the degree 3 boundary edge, thus having a unique interior edge of degree four. This is $X_2^2,$ and there may be further identifications from other face pairings in $\Phi.$

In the second case, first assume that the two pre-images of $e$ are contained in a face of $\widetilde{\sigma}_1.$ This forces $\widetilde{\sigma}_2$ to also have a face containing two pre-images of $e,$ and these two faces are identified. Now there are two other faces of $\widetilde{\sigma}_1$ containing a pre-image of $e.$ If they are identified, then a face is a cone. Hence both are identified with faces of $\widetilde{\sigma}_2.$ There is a unique way to do this under the condition that $e$ has degree four. But then the remaining free face of each $\widetilde{\sigma}_i$ is a cone or a dunce hat. In either case, this forces $M=S^3,$ contradicting minimality.

Hence assume that the two pre-images of $e$ are not contained in a face of either $\widetilde{\sigma}_1$ nor  $\widetilde{\sigma}_2.$ Then the triangulation contains precisely two tetrahedra. First assume that two faces of $\widetilde{\sigma}_1$ are identified. Then the same is true for $\widetilde{\sigma}_2.$ Analysing all possible identifications gives (up to combinatorial equivalence) the triangulation $X_2^0$ with the specified marked edge of degree four. The triangulation can be viewed as two solid tori identified along their boundary, and the corresponding lens space identified as $L(8,3).$ Next, assume that no two faces of $\widetilde{\sigma}_1$ are identified. This yields the triangulation $X_2^1$ of with a marked edge. However, all three edges are combinatorially equivalent. To identify the manifold, notice that for each edge, there is a normal surface made up of two quadrilaterals, one in each tetrahedron, which don't meet that edge. This is an embedded Klein bottle and a one-sided Heegaard surface for the manifold. Using \cite{Rubin1978}, the manifold can now be identified as $S^3/Q_8.$

If $|\widetilde{\Delta}_e| = 3,$ denote the tetrahedra in $\widetilde{\Delta}_e$ by $\widetilde{\sigma}_1,$ $\widetilde{\sigma}_2$ and $\widetilde{\sigma}_3.$ Without loss of generality, $\widetilde{\sigma}_2$ contains precisely two pre-images of $e.$ First assume that there is no face pairing between faces of $\widetilde{\sigma}_1$ and $\widetilde{\sigma}_3.$ Then the two faces incident with the pre-image of $e$ in $\widetilde{\sigma}_i,$ $i=1,3,$ are identified with faces of $\widetilde{\sigma}_2.$ Analysing the possibilities gives $X_3^1,$ and there may be further identifications from other face pairings in $\Phi.$ Next assume that there is a pairing between faces of $\widetilde{\sigma}_1$ and $\widetilde{\sigma}_3$ containing a pre-image of $e.$ Then there is a unique such pairing, and the remainder of the argument is as in the proof of (3c), giving $X_3^0.$

If $|\widetilde{\Delta}_e| = 4,$ then the only possibility is the octahedron with possible identifications along its boundary.
\end{proof}

The simplicial maps from $X_{4;2}^2$ to a minimal, 0--efficient triangulation of a closed, orientable 3--manifold are $X_{4;2}^2 \to L(7,2)$ and $X_{4;2}^2 \to L(8,3)= X_{4;2}^0.$ The triangulation $X_{4;2}^0$ of $L(8,3)$ contains three edges of degree four; two of them have neighbourhoods modelled on $X_{4;2}^2.$ Thus, we have classified all minimal, 0--efficient triangulations with at most two tetrahedra containing an edge of degree four.

It is also true that $\lst_2 \to L(7,2),$ but $\lst_2$ has no interior edge of degree four, and hence does not appear in the above list.



\subsection{Edges of degree equal to five}

\begin{proposition}[Edges of degree five]\label{pro:degree five edges}
Assume the minimal and 0--efficient triangulation $\tri$ of the closed, orientable, connected and irreducible 3--manifold $M$ has an edge of degree five, denoted $\overline{e}.$ Then the model, $(X_{5;k}, e),$ for $\overline{e}$ is one of the following:
\begin{enumerate}
\item[(5a)] If $k=1,$ then $X_{5;1}$ is a one-tetrahedron triangulation of $S^3$ which contains a (necessarily unique) edge of degree five, $e.$ In particular, $M=S^3.$
\item[(5b)] If $k=2,$ then $X_{5;2}=X_{5;2}^0$ or $X_{5;2}^1,$ where $X_{5;2}^0$ is a two-tetrahedron triangulation of $L(3,1)$ with $e$ either of the two edges of degree five; and $X_{5;2}^1$ is a two-tetrahedron triangulation of $L(7,2)$ with $e$ the unique edge of degree five.
\item[(5c)] If $k=3,$ then $X_{5;3}=X_{5;3}^0,$ $X_{5;3}^1$ or $X_{5;3}^2,$ where $X_{5;3}^0$ is $\lst_1$ with two tetrahedra attached to its boundary faces such that the degree 3 boundary edge becomes an interior edge of degree five, $e;$ $X_{5;3}^1=\{5,3,8\}$ is $X^2_{4;2}=\{5,2,3\}$ with another tetrahedron layered on the degree 4 boundary edge, giving a unique interior edge of degree five, $e;$ and $X_{5;3}^2$ is a solid torus with a unique interior edge of degree five, $e,$ obtained from a 3-tetrahedron triangulated prism by identifying two boundary squares.
\item[(5d)] If $k=4,$ then $X_{5;4}=X_{5;4}^0$ or $X_{5;4}^1,$ where $X_{5;4}^0$ is $\lst_1$ with three tetrahedra attached to its boundary faces such that the degree 2 boundary edge becomes an interior edge of degree five, $e;$ and $X_{5;4}^1$ is a solid torus obtained by identifying a pair of opposite edges of a tetrahedron, and then layering one tetrahedron on the resulting edge, and attaching two tetrahedra to the remaining boundary faces to create a unique interior edge of degree five, $e.$
\item[(5e)] If $k=5,$ then $X_{5,5}$ is a ball triangulated with five tetrahedra such that their intersection is a unique interior edge of degree five, $e.$
\end{enumerate}
The maximal complexes are shown in Figure \ref{fig:degree five edges}.
\end{proposition}

\begin{figure}[t]
\psfrag{M}{{\small $X^2_{4;2}$}}
\psfrag{N}{{\small $\lst_1$}}
\psfrag{e}{{\small $e$}}
\begin{center}
    \subfigure[$X_{5;1} \cong S^3$]{\label{fig:X5_1}
      \includegraphics[height=2.1cm]{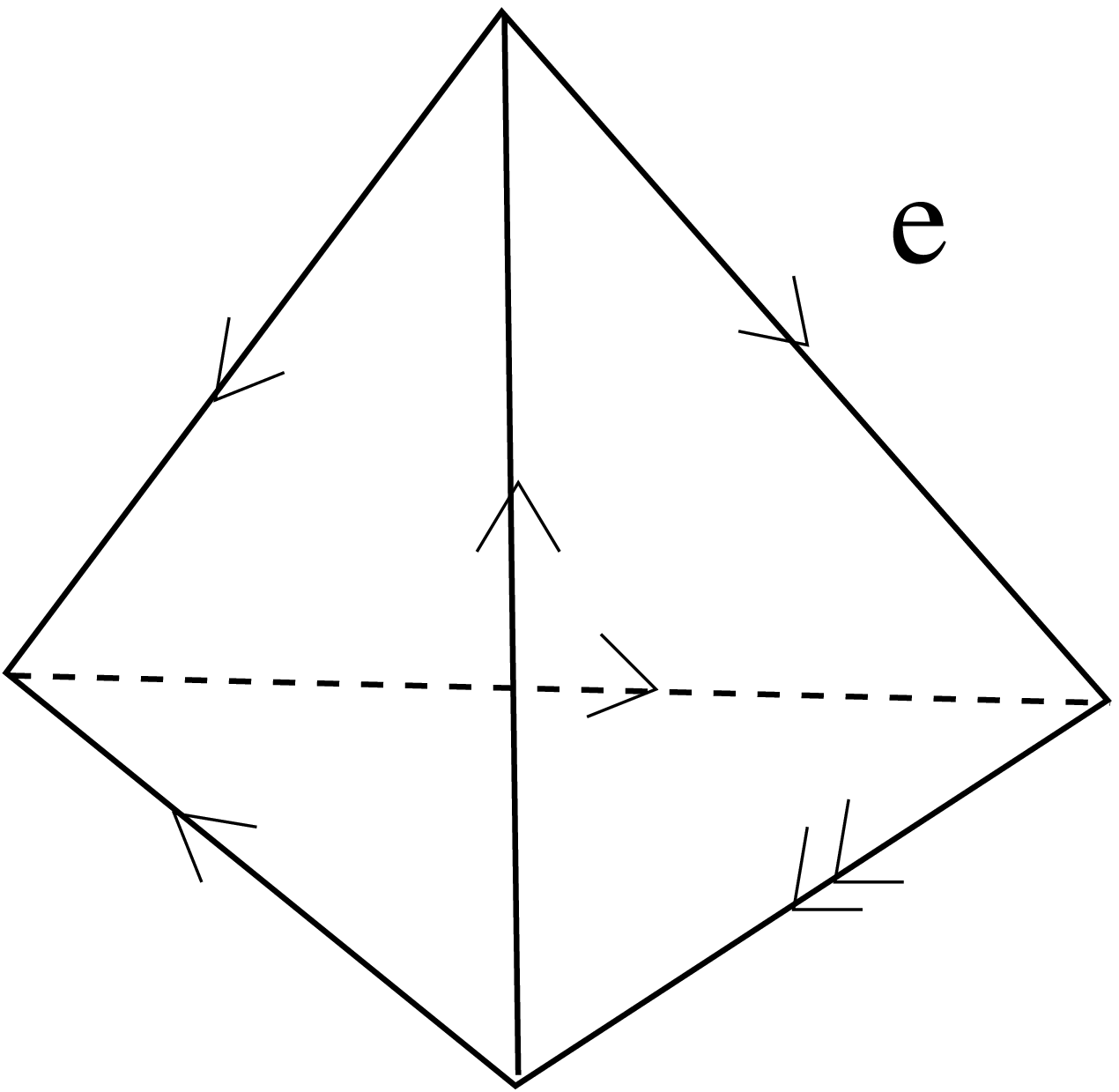}
    } 
    \quad
    \subfigure[$X^0_{5;2} \cong L(3,1)$]{\label{fig:X5_20}
      \includegraphics[height=2.1cm]{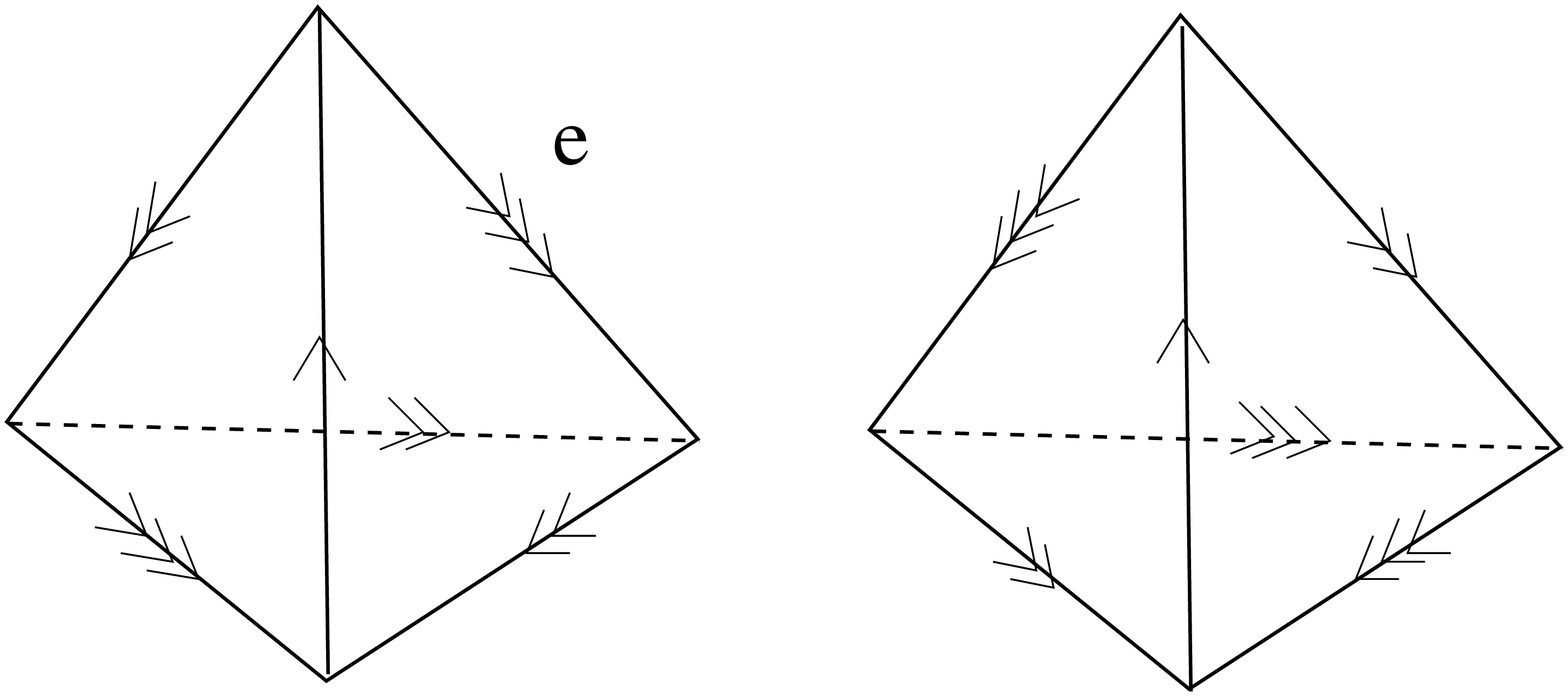}
    } 
   \quad
           \subfigure[$X^1_{5;2}\cong L(7,2)$]{\label{fig:X5_21}
      \includegraphics[height=2.1cm]{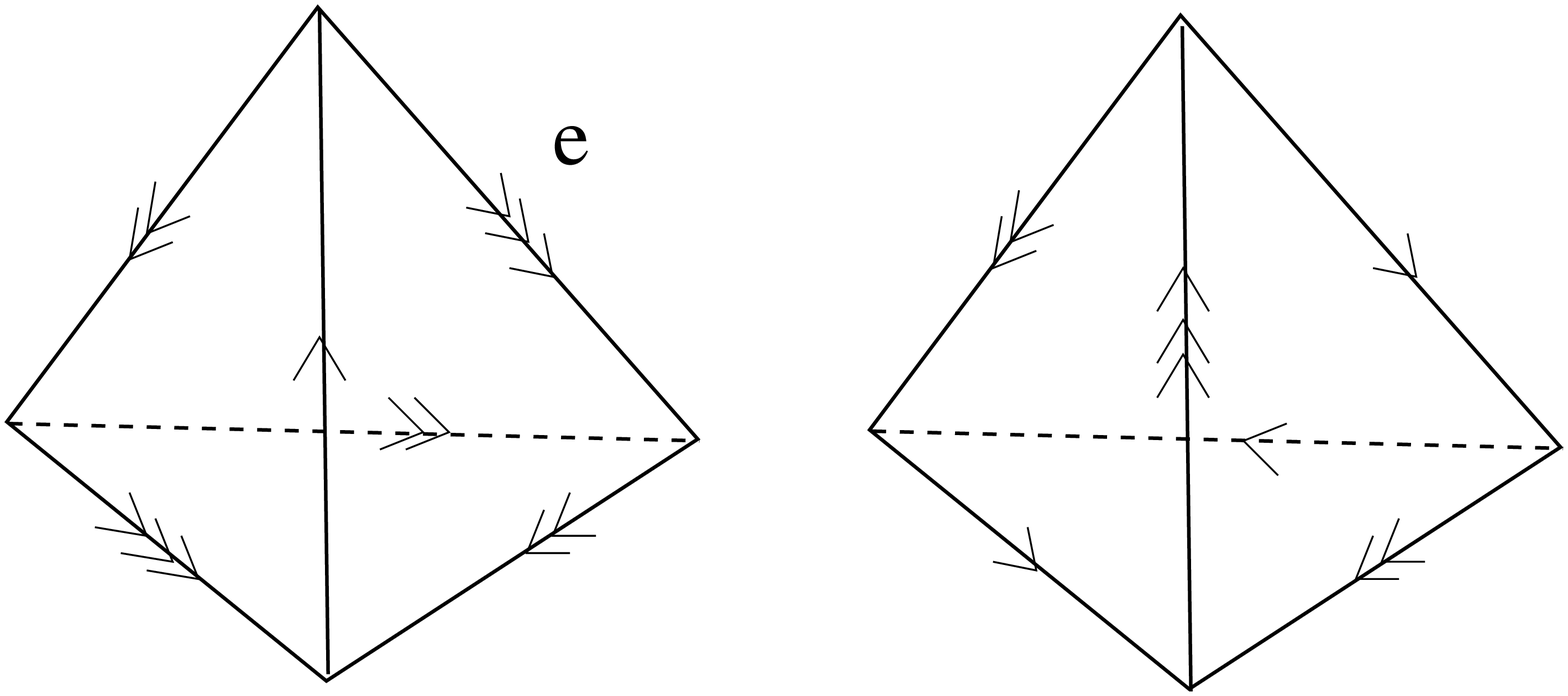}
    } 
   \\
           \subfigure[$X_{5;3}^0\cong$ solid torus]{\label{fig:X5_30}
      \includegraphics[height=2.3cm]{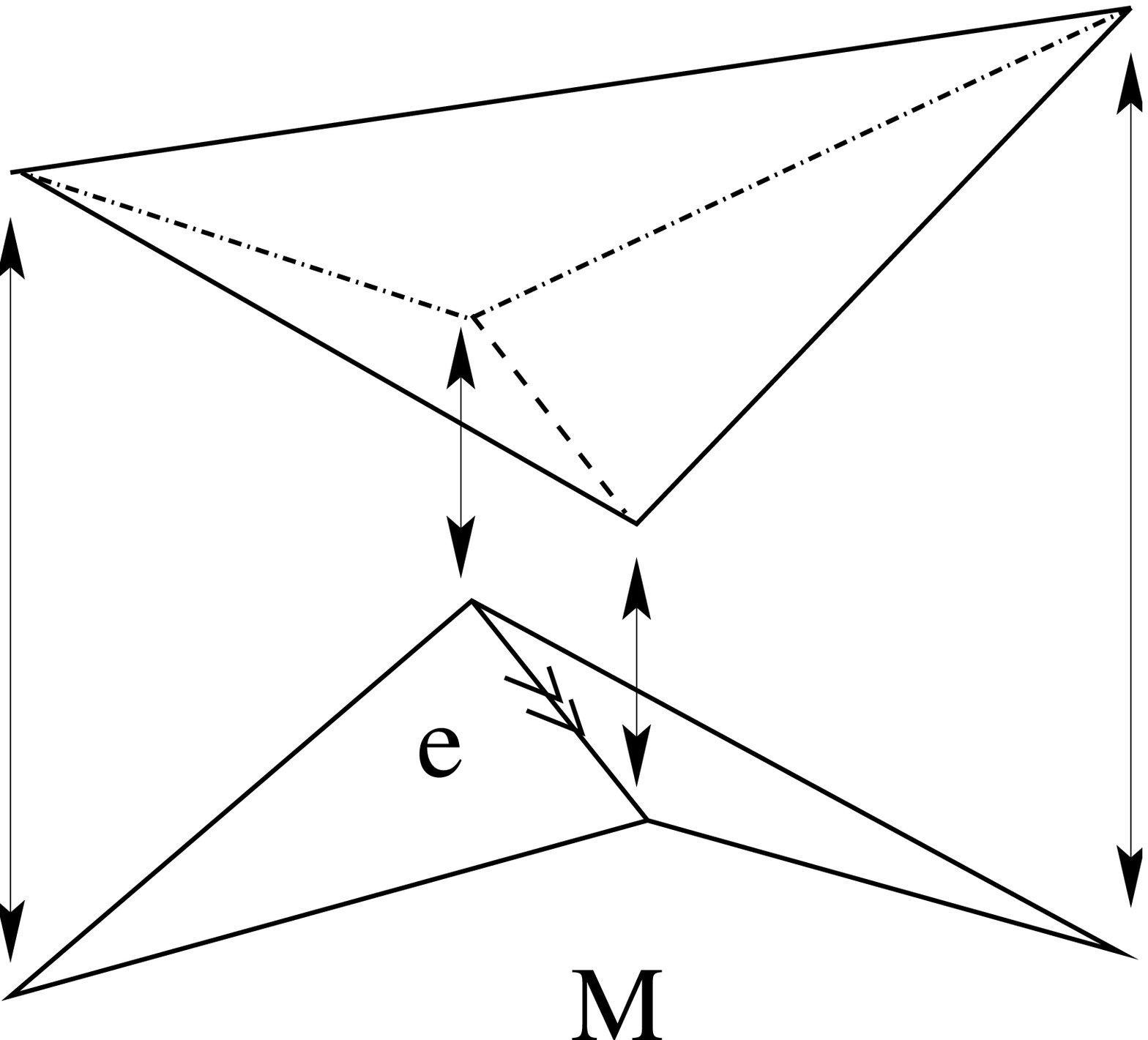}
    } 
    \quad
           \subfigure[$X_{5;3}^1\cong$ solid torus]{\label{fig:X5_31}
      \includegraphics[height=2.3cm]{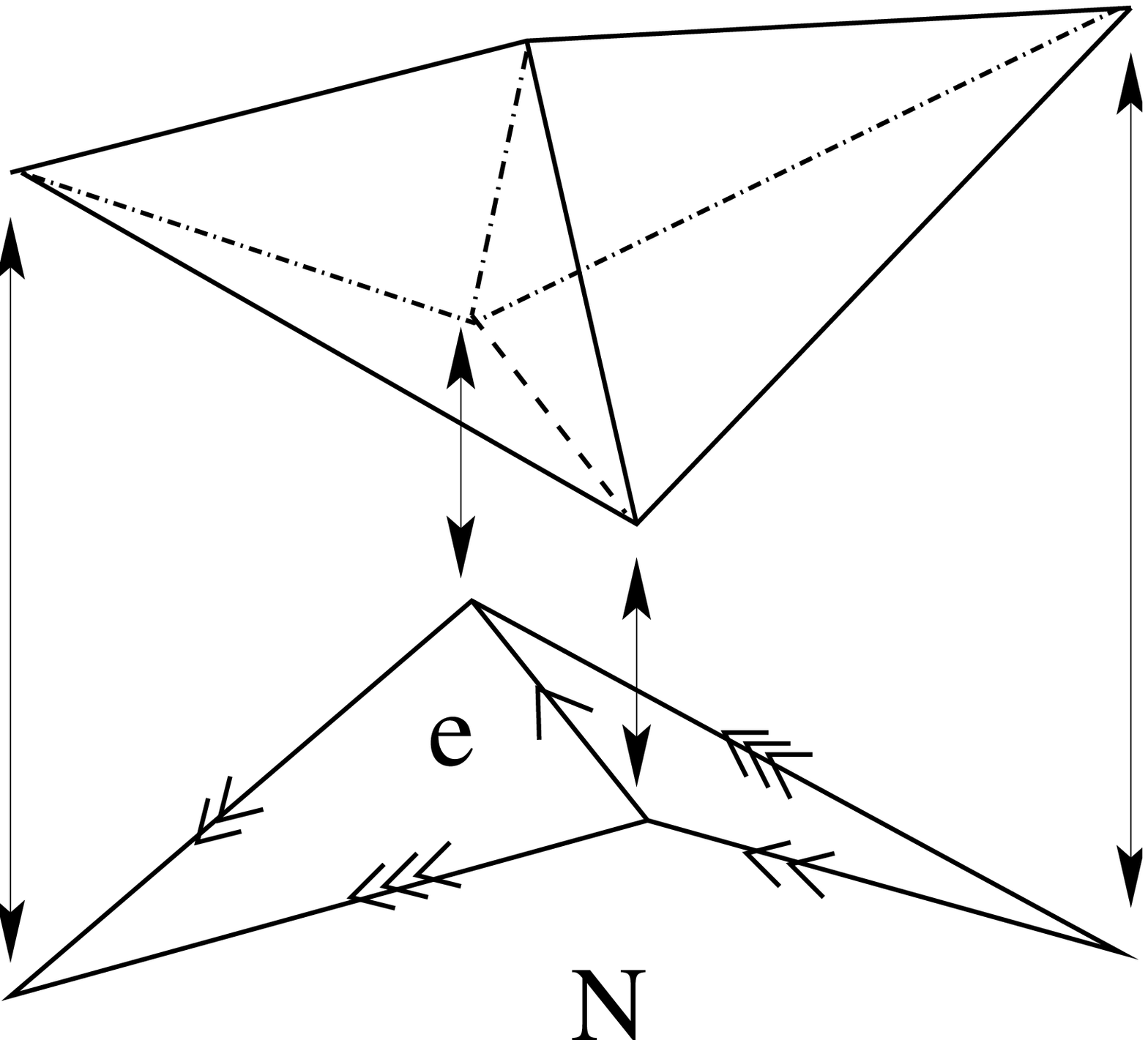}
    } 
    \quad
    \subfigure[$X^2_{5;3}\cong$ solid torus]{\label{fig:X5_32}
      \includegraphics[height=2.2cm]{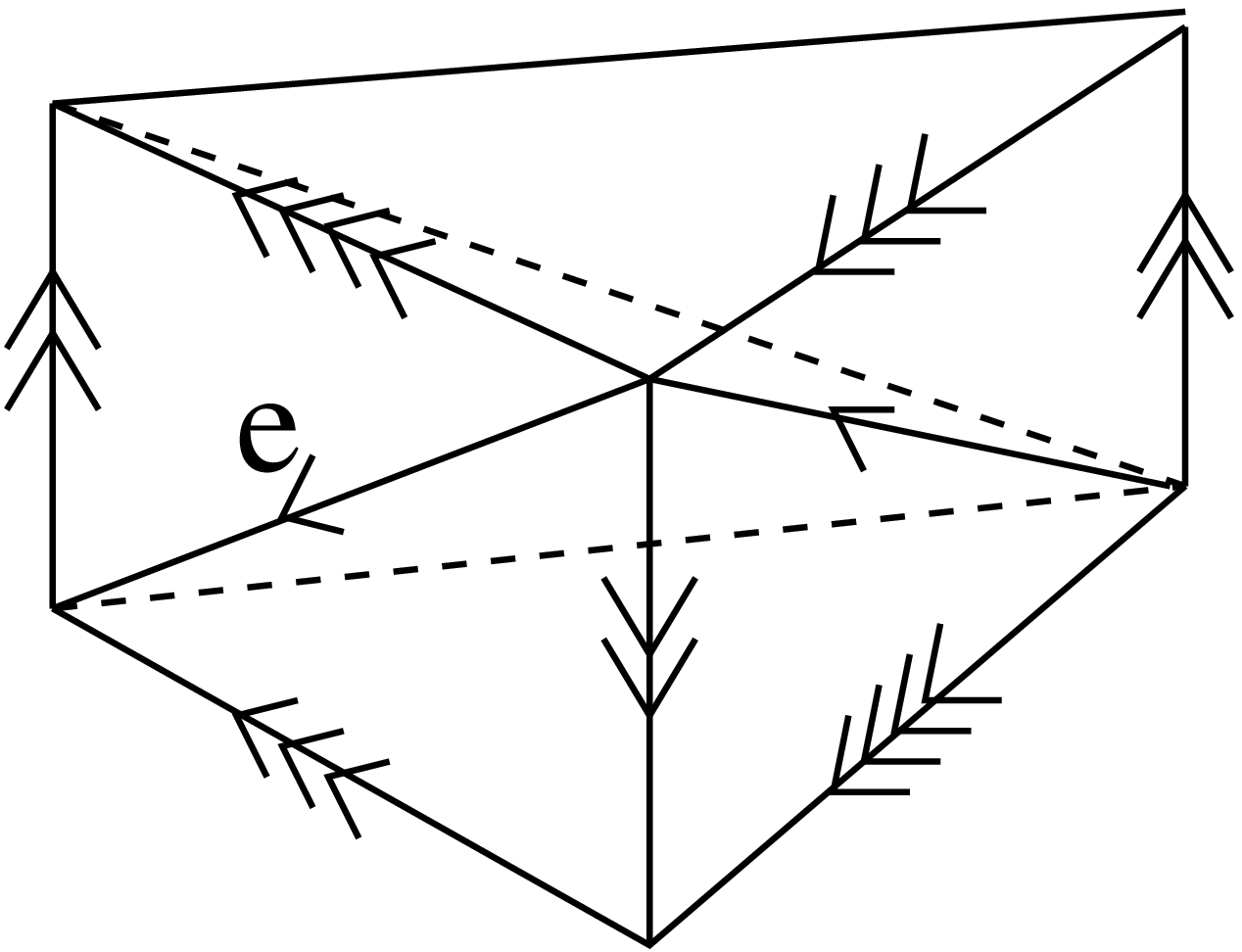}
    } 
    \quad
    \subfigure[$X^0_{5;4}\cong$ solid torus]{\label{fig:X5_40}
      \includegraphics[height=2.4cm]{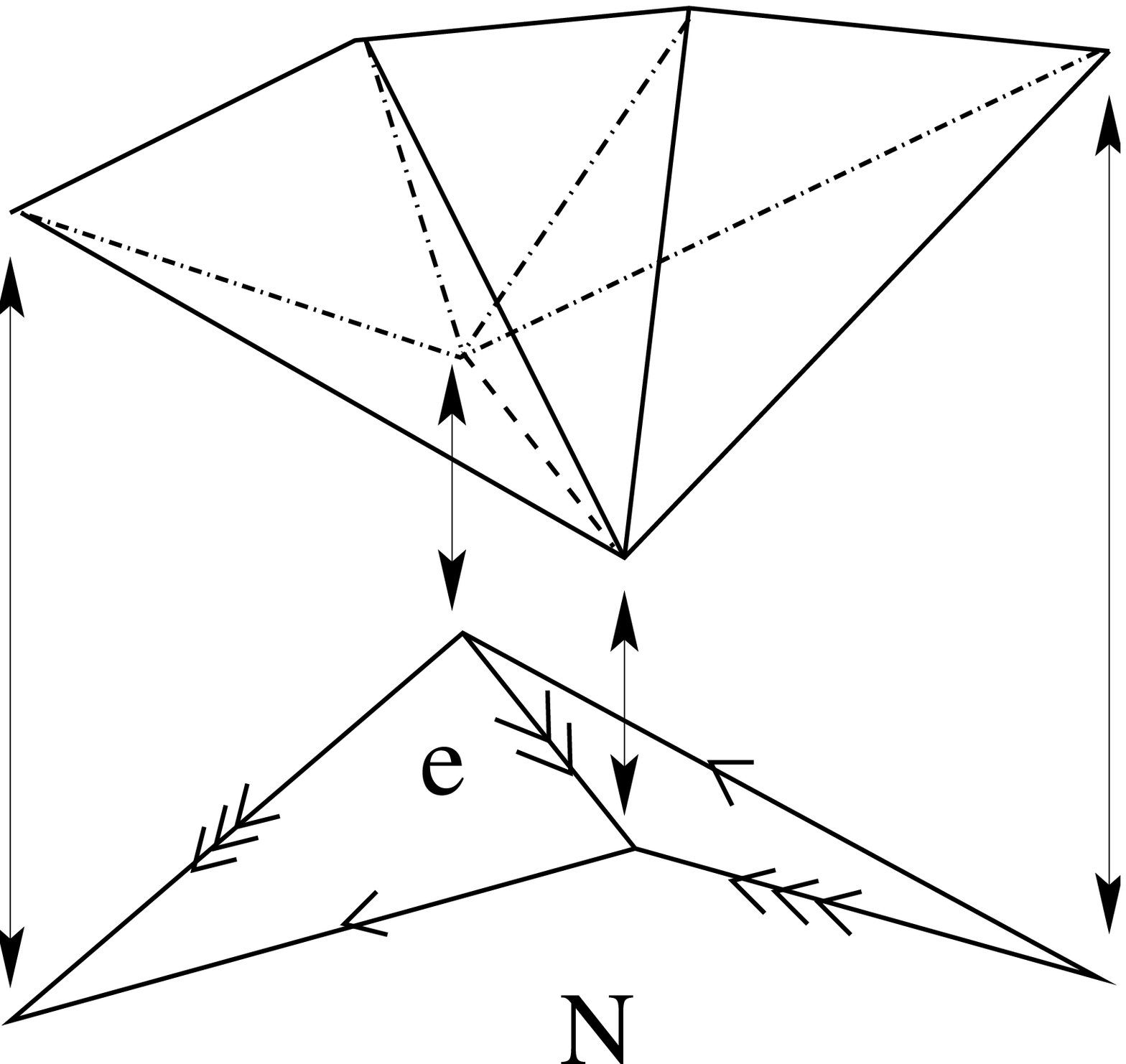}
    } 
     \\
    \subfigure[$X^1_{5;4}\cong$ solid torus]{\label{fig:X5_41}
      \includegraphics[height=2.3cm]{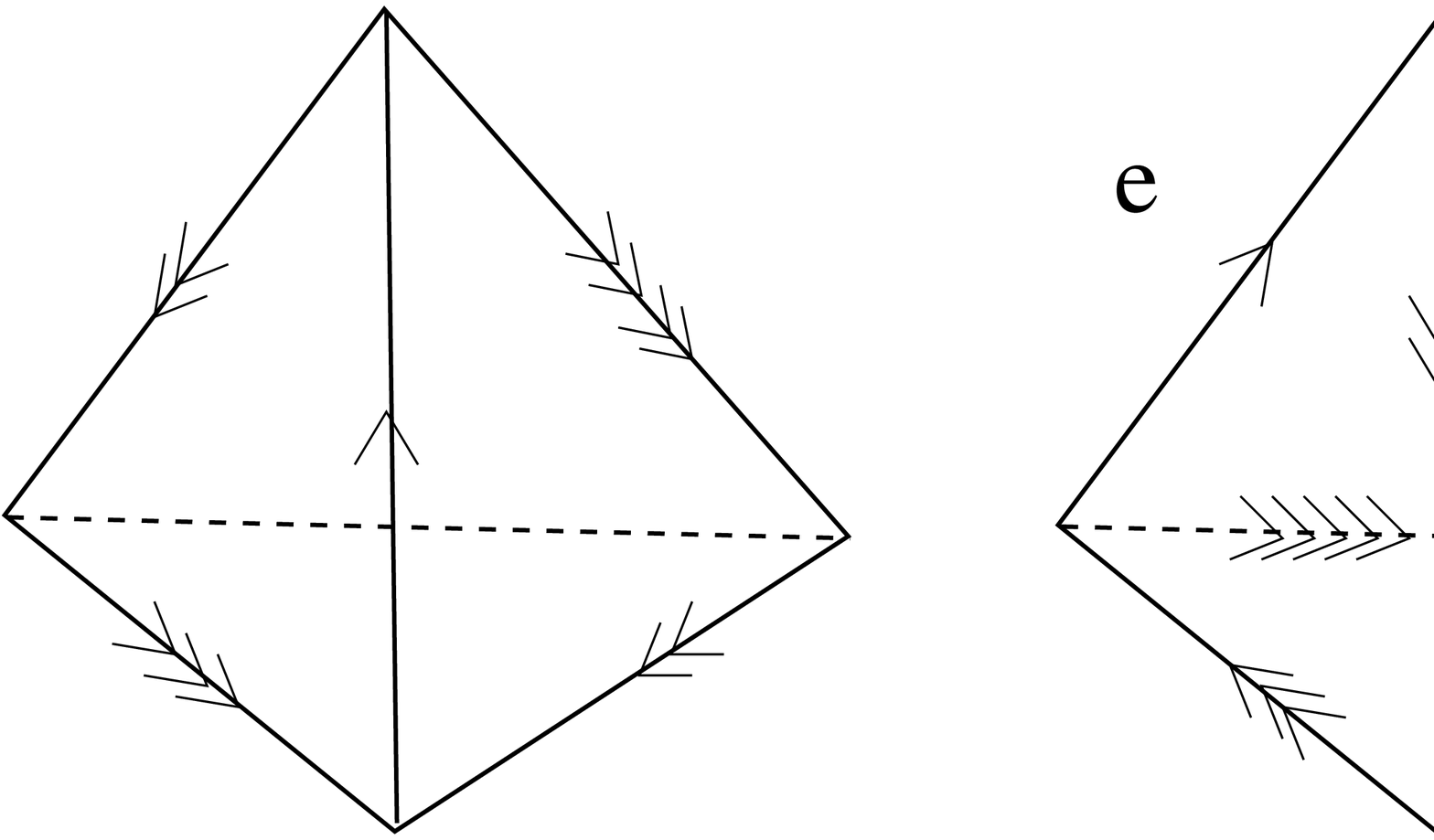}
    } 
     \qquad
         \subfigure[$X_{5;5}\cong$ 3--ball]{\label{fig:X5_5}
      \includegraphics[height=2.5cm, width=3cm]{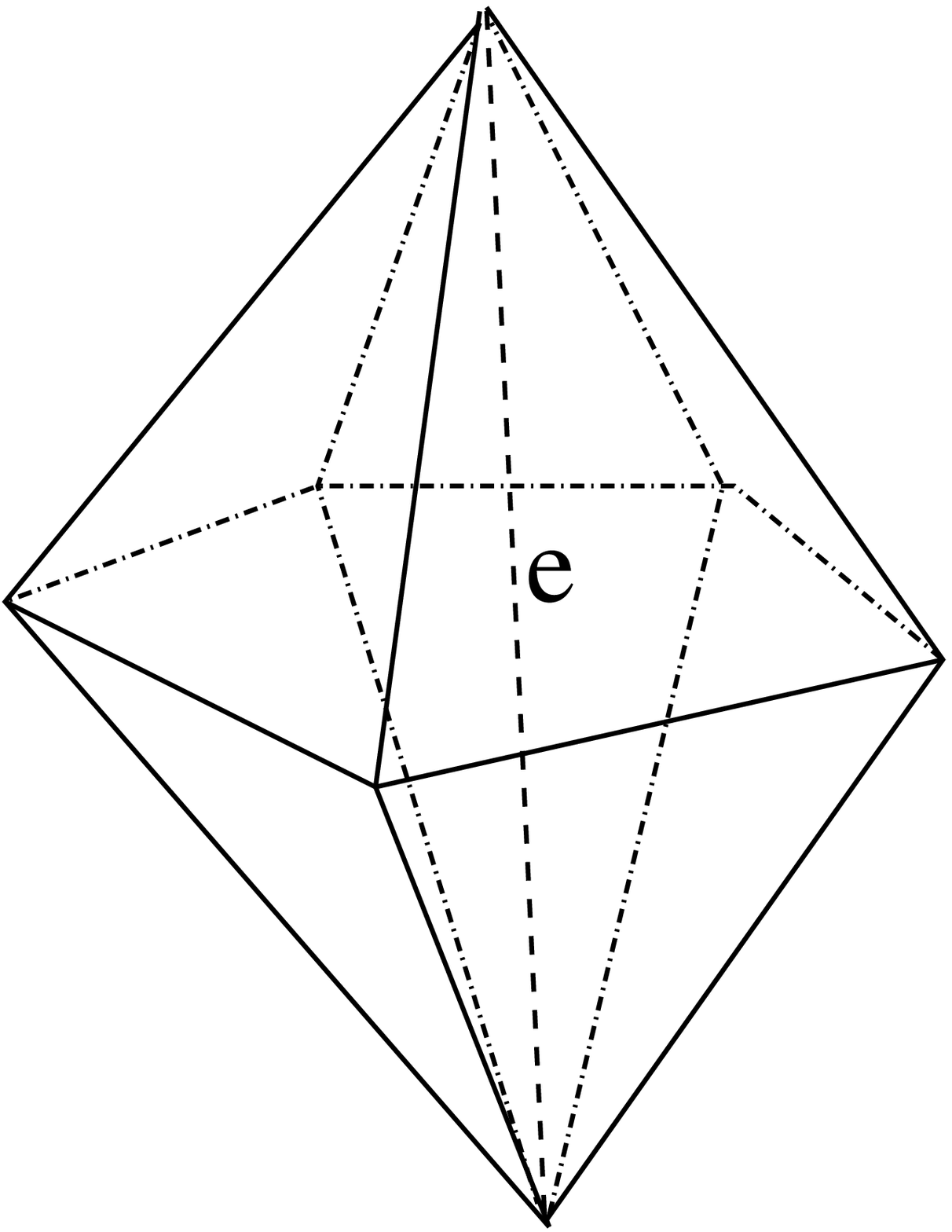}
    } 
\end{center}
    \caption{Degree five edges in minimal and 0--efficient triangulations}
     \label{fig:degree five edges}
\end{figure}

\begin{proof}
The proof follows the same line of argument as the previous two propositions' proofs. We only highlight the main points that are different. Firstly, in the case where $|\widetilde{\Delta}_e| = 2$ and one arrives at the triangulation of $L(3,1),$ minimality does not imply 0--efficiency. The latter property can be verified by computing the set of all connected normal surfaces of Euler characteristic equal to two.

In the case where $|\widetilde{\Delta}_e| = 2,$ there is a subcase where $\widetilde{\sigma}_1$ contains one pre-image of $\overline{e},$ and $\widetilde{\sigma}_2$ and $\widetilde{\sigma}_3$ contain two each. If two pre-images are contained on a common face, then one obtains $X^2_{5;3}.$ Otherwise there are two subsubcases. First, one assumes that each of $\widetilde{\sigma}_2$ and $\widetilde{\sigma}_3$ has two of its faces identified. Then each is equivalent to $\lst_1,$ and these two subcomplexes must meet in a face. One thus obtains a pinched 2--sphere made up of the two faces of $\widetilde{\sigma}_1$ which meet $\widetilde{\sigma}_2$ and $\widetilde{\sigma}_3$ respectively. This is not possible due to minimality. Hence assume that
precisely one of $\widetilde{\sigma}_2$ and $\widetilde{\sigma}_3$ is combinatorially equivalent to $\lst_1.$ Analysing the possibilities gives $X^0_{5;3}.$ Last, assume that none of $\widetilde{\sigma}_2$ and $\widetilde{\sigma}_3$ has two of its faces identified. Analysing all possible gluings of the remaining faces, one obtains a 3--tetrahedron complex whose boundary either consists of two faces which form a pinched 2--sphere, or one of whose boundary faces is a cone or a dunce hat. In either case, one obtains a contradiction.
\end{proof}


{

\section{Normal surfaces dual to $\mathbf{\Z_2}$--cohomology classes}
\label{sec:co-homology classes}

Throughout this section, let $\tri$ be an arbitrary 1--vertex triangulation of the closed 3--manifold $M,$ and $\varphi \co \pi_1(M) \to \Z_2$ be a non--trivial homomorphism. Additional hypotheses will be stated. A colouring of edges arising from $\varphi$ is introduced and a canonical normal surface dual to $\varphi$ is determined. This yields a combinatorial constraint on the triangulation, which is then specialised to a family of lens spaces.


\subsection{Colouring of edges and dual normal surface}
\label{subsec:colouring}

\begin{figure}[t]
\begin{center}
      \includegraphics[height=3cm]{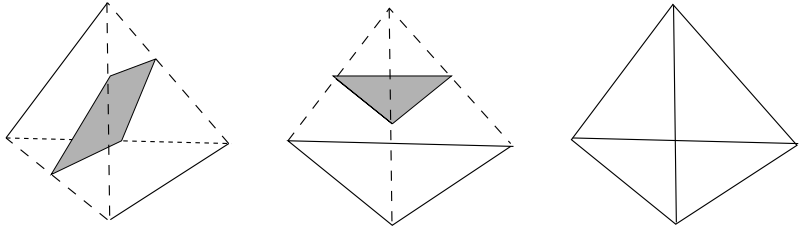}
\end{center}
    \caption{Types of tetrahedra and normal discs of the dual surface}
     \label{fig:Z2-homology class}
\end{figure}

Each edge, $e,$ is given a fixed orientation, and hence represents an element $[e]\in \pi_1(M).$ If $\varphi[e]=0,$ the edge is termed $\varphi$--even, otherwise it is termed $\varphi$--odd. This terminology is independent of the chosen orientation for $e.$ Faces in the triangulation give relations between loops represented by edges. It follows that a tetrahedron falls into one of the following categories, which are illustrated in Figure~\ref{fig:Z2-homology class}:
\begin{itemize}
\item[] Type 1: A pair of opposite edges are $\varphi$--even, all others are $\varphi$--odd.
\item[] Type 2: The three edges incident to a vertex are $\varphi$--odd, all others are $\varphi$--even.
\item[] Type 3: All edges are $\varphi$--even.
\end{itemize}

It follows from the classification of the tetrahedra in $\tri$ that, if $\varphi$ is non-trivial, then one obtains a unique normal surface, $S_\varphi(\tri),$ with respect to $\tri$ by introducing a single vertex on each $\varphi$--odd edge. This surface is disjoint from the tetrahedra of type 3; it meets each tetrahedron of type 2 in a single triangle meeting all $\varphi$--odd edges; and each tetrahedron of type 1 in a single quadrilateral dual to the $\varphi$--even edges. Moreover, $S_\varphi(\tri)$ is dual to the $\Z_2$--cohomology class represented by $\varphi.$


\subsection{Combinatorial bounds for triangulations}

The set-up and notation of the previous subsection is continued. Let
\begin{itemize}
\item[] $A(\tri)=$ number of tetrahedra of type 1,
\item[] $B(\tri)=$ number of tetrahedra of type 2,
\item[] $C(\tri)=$ number of tetrahedra of type 3,
\item[] $\odd(\tri)=$ number of \emph{$\varphi$--odd} edges,
\item[] $\even(\tri)=$ number of \emph{$\varphi$--even} edges,
\item[] $\tilde{\even}(\tri)=$ number of pre-images of $\varphi$--even edges in $\widetilde{\Delta}.$
\end{itemize}
Assume that $\tri$ contains $T(\tri)$ tetrahedra. Then $T(\tri)= A(\tri)+B(\tri)+C(\tri).$ For the remainder of this subsection, we will write $A = A(\tri),$ etc. The complex $K$ in $M$ spanned by all $\varphi$--even edges in $\tri$ is homotopy equivalent to the complement, $N,$ of a regular neighbourhood of $S_\varphi$ in $M.$ We therefore have:
$$
2 \chi (S_\varphi)  = \chi (\partial N) = 2 \chi (N) = 2 \chi (K).
$$
The Euler characteristic of $K$ can be computed from the combinatorial data. We have:
$$
\chi (K) = 1 - \even + \frac{B+4C}{2} - C,
$$
and hence:
\begin{equation}\label{eq:e}
2C + B = 2 \even - 2 + 2  \chi(S_\varphi).
\end{equation}
So:
\begin{align}\label{inequ for one vert tri gen}
\tilde{\even} 	&= 2A+3B+6C \nonumber \\
		&=  2A + B + 2C + 4 \even -4 + 4 \chi(S_\varphi)\nonumber\\
		&\le 2T - 4   + 4 \chi(S_\varphi) + 4 \even.
\end{align}
\begin{lemma}
Let $M$ be a closed, orientable, irreducible 3--manifold which is not homeomorphic to one of $\R P^3,$ $L(4,1),$ and $\varphi \co \pi_1(M) \to \Z_2$ be a non--trivial homomorphism. Suppose that $\tri$ is a minimal triangulation with $T$ tetrahedra, and let $S_\varphi$ be the canonical normal surface dual to $\varphi.$ Then
\begin{equation}\label{inequ for min tri general}
\even_3 \ge 4 - 2T - 4 \chi(S_\varphi) + \sum_{j=5}^{\infty} (j-4) \even_j,
\end{equation}
where $\even_i$ is the number of $\varphi$--even edges of degree $i.$
\end{lemma}

\begin{proof}
First note that the existence of $\varphi$ implies that $M \neq S^3, L(3,1).$ It now follows from \cite{JR}, Theorem 6.1, that $\tri$ has a single vertex; hence $S_\varphi$ and $\even_i$ are defined. Moreover, \cite{JR}, Proposition 6.3 (see also Proposition \ref{pro:degree one and two edges}), implies that the smallest degree of an edge in $\tri$ is three. One has $\tilde{\even} = \sum i \even_i$ and $\even = \sum \even_i.$ Putting this into (\ref{inequ for one vert tri gen}) gives the desired inequality.
\end{proof}


\subsection{Combinatorial bounds for certain lens spaces}
\label{subsec:comb bds}

Let $M=L(2n,q)$ with a minimal triangulation, $\tri,$ and assume $M \neq \R P^3$ or $L(4,1).$ There is a unique non--trivial homomorphism $\varphi \co \pi_1 M \to \Z_2.$ It is known through work by Bredon and Wood \cite{BW} and the second author \cite{Rubin1978} that, up to isotopy, there is a unique incompressible non-orientable surface $S_0$ in $L(2n,q).$ The surface $S_\varphi(\tri)$ must be non-orientable since it is dual to a non-trivial $\Z_2$--cohomology class and there is no non-trivial $\Z$--cohomology class. It therefore compresses to the surface $S_0,$ so 
$$\chi(S_0) \ge \chi (S_\varphi(\tri)).$$
The smallest degree of an edge in $\tri$ is three. Denote by $\lay_{2n,q}$ the minimal layered triangulation of $M=L(2n,q).$ Then Theorem~8.2 of \cite{JR:LT} implies that 
$$\chi (S_\varphi(\tri)) \le \chi(S_0) = \chi (S_\varphi(\lay_{2n,q}))= 1 - \even(\lay_{2n,q}).$$
We also have that the number of tetrahedra in $\tri$ satisfies 
$$ T(\tri) \le \even(\lay_{2n,q}) + \odd(\lay_{2n,q}) - 1.$$
Hence
\begin{equation}\label{inequ:odd and even}
4 - 2T(\tri) - 4 \chi(S_\varphi(\tri)) \ge 2 + 2 \even(\lay_{2n,q}) - 2 \odd(\lay_{2n,q}).
\end{equation}
Now \emph{assuming} that $\even(\lay_{2n,q}) \ge \odd(\lay_{2n,q}),$ inequality (\ref{inequ for min tri general}) becomes:
\begin{equation}\label{inequ for min tri}
\even_3 \ge 2 + \sum_{j=5}^{\infty} (j-4) \even_j.
\end{equation}
We will show that this inequality forces $\tri=\lay_{2n,q}$ as well as $\even(\lay_{2n,q}) = \odd(\lay_{2n,q}).$

The set of all lens spaces $L(2n,q)$ having minimal layered triangulations satisfying $\even(\lay_{2n,q}) \ge \odd(\lay_{2n,q})$ is non-empty: an edge in the triangulation $\lay_{2n,1}=\lay_k$ of $L(2n,1)=L(k+3,1)$ is $\varphi$--even if and only if it is even; and hence odd if and only if it is $\varphi$--odd. Whence $\even(\lay_k) = \odd(\lay_k).$


\section{Intersections of maximal layered solid tori}
\label{sec:Intersections of maximal layered solid tori}

Throughout this section, assume that $M$ is an irreducible, orientable, connected 3--manifold with a fixed triangulation, $\tri.$ Further assumptions will be stated. We prove some general facts about layered solid tori and maximal layered solid tori in such a triangulation.


\subsection{Layered solid tori}

\begin{definition}(Layered solid torus)
A \emph{layered solid torus with respect to $\tri$} in $M$ is a subcomplex in $M$ which is combinatorially equivalent to a layered solid torus. Any reference to $\tri$ is suppressed when $\tri$ is fixed. The edges of the layered solid torus $T$ are termed as follows: If there is more than one tetrahedron, then there is a unique edge which has been layered on first, termed the \emph{base-edge (of $T$).} The edges in the boundary of $T$ are called \emph{boundary edges (of $T$)} and all other edges (including the base-edge) are termed  \emph{interior edges (of $T$).} 
\end{definition}

For every edge in $M,$ we will also refer to its degree as its $M$--degree, and its degree with respect to the layered solid torus $T$ in $M$ is called its $T$--degree. The $T$--degree of an edge not contained in $T$ is zero. The unique edge of $T$--degree one is termed a \emph{univalent edge (for $T$)}. Clearly, $T$--degree and $M$--degree agree for all interior edges of $T.$

\begin{lemma}\label{lem:meet in 2 faces}
If the intersection of two layered solid tori consists of two faces, then $M$ is a lens space with layered triangulation.
\end{lemma}

\begin{proof}
If two layered solid tori, $T_1$ and $T_2,$ meet in precisely two faces, then these must be their respective boundary faces. Assume that the identification of the faces extends to a homeomorphism of the boundary tori. Whence $M$ is a lens space. If tetrahedron $\sigma$ in $T_1$ meets $T_2,$ then $\sigma$ is layered on a boundary edge of $T_2$ and $T_2 \cup \sigma$ is a layered solid torus unless there are further identifications of the two faces of $\sigma$ not meeting $T_2.$ But this forces $\sigma = T_1$ and $M$ is a lens space with layered triangulation. Otherwise add $\sigma$ to $T_2$ and subtract it from $T_1;$ the result now follows inductively.

Hence assume that the identification of the faces does not extend to a homeomorphism of the boundary tori. Examining the resulting face pairings (using the fact that $M$ is orientable), one observes that an edge folds back onto itself, which is not possible.
\end{proof}

\begin{lemma} \label{lem:meet in tet}
If two layered solid tori share a tetrahedron, then either one contains the other or $M$ is a lens space with layered triangulation.
\end{lemma}

\begin{proof}
Assume that one is not contained in the other, and denote the layered solid tori by $T_1$ and $T_2.$ Then the closure of $T_1 \setminus (T_1 \cap T_2)$ is non-empty and a layered solid torus. It meets $T_2$ in precisely two faces and hence Lemma \ref{lem:meet in 2 faces} yields the result.
\end{proof}

\begin{lemma}\label{lem:layered solid tori don't meet in single face}
If the triangulation is minimal and 0--efficient, then the intersection of two layered solid tori cannot consist of a single face.
\end{lemma}

\begin{proof}
If $T_1$ and $T_2$ meet in a single face, then the free faces give rise to a pinched 2--sphere in $M.$ The boundary of a regular neighbourhood of $T_1 \cup T_2$ in $M$ is an embedded 2--sphere. As in the proof of Proposition~\ref{pro:degree three edges}, Case (5), a barrier argument gives a contradiction to minimality.
\end{proof}

\begin{lemma}
If the triangulation is minimal and 0--efficient, then the intersection of two layered solid tori cannot consists of three edges.
\end{lemma}

\begin{proof}
If $T_1$ and $T_2$ meet in precisely three edges, then the boundary of a regular neighbourhood of $T_1 \cup T_2$ in $M$ consists of two embedded 2--spheres. As in the above lemma, this can be evacuated, resulting in a smaller triangulation.
\end{proof}

\begin{lemma}
If the triangulation is minimal and 0--efficient, and the intersection of two layered solid tori consists of two edges, then $M$ is a lens space and the triangulation is a minimal layered triangulation.
\end{lemma}

\begin{proof}
If $T_1$ and $T_2$ meet in two edges, they form a spine for the boundary of each solid torus. The boundary of a regular neighbourhood of $T_1 \cup T_2$ in $M$ is an embedded 2--sphere. As above, this 2--sphere shrinks to a point in the complement of $T_1 \cup T_2.$ Hence there is a homotopy in $M$ identifying the boundaries of $T_1$ and  $T_2.$ If this homotopy identifies the remaining edges, then one obtains a new triangulation having fewer tetrahedra. Hence assume that the homotopy does not identify the remaining edges. Then a triangulation is obtained by inserting a single tetrahedron. The complement of $T_1 \cup T_2$ consists therefore of exactly one tetrahedron and the triangulation is a minimal layered triangulation.
\end{proof}


\subsection{Maximal layered solid tori}

\begin{definition}(Maximal layered solid torus)
A layered solid torus is a \emph{maximal layered solid torus with respect to $\tri$} in $M$ if it is not strictly contained in any other layered solid torus in $M.$
\end{definition}

A layered triangulation of a lens space contains precisely two maximal layered solid tori. The lemmata of the previous section directly imply the following:

\begin{lemma}\label{lem:intersection of maximal layered solid tori}
Assume that the triangulation is minimal and 0--efficient. If $M$ is not a lens space with layered triangulation, then the intersection of two distinct maximal layered solid tori in $M$ consists of at most a single edge.
\end{lemma}

\begin{lemma}\label{lem:layered lens char 2}
Assume that the triangulation is minimal and 0--efficient, and suppose that $M$ contains a layered solid torus, $T,$ made up of at least two tetrahedra and having a boundary edge, $e,$ which has degree four in $M.$ Then either
\begin{enumerate}
\item $T$ is not a maximal layered solid torus in $M;$ or
\item $e$ is the univalent edge for $T$ and it is contained in four distinct tetrahedra in $M;$ or 
\item $M$ is a lens space with minimal layered triangulation.
\end{enumerate}
\end{lemma}

\begin{proof}
We apply Proposition \ref{pro:degree four edges} to identify the complex $X_{4;k}$ on which the neighbourhood of $e$ is modelled. Since $T$ contains at least two tetrahedra, this rules out $X_{4;1}^0,$ $X_{4;1}^1,$ $X_{4;2}^0$ and $X_{4;2}^1.$

If the neighbourhood is modelled on $X_{4;2}^2,$ then $M$ is decomposed into $T$ and the one-tetrahedron solid torus, $\lst_1.$ Hence $M$ is a lens space with layered triangulation. By assumption, this must be a minimal layered triangulation.

If the neighbourhood is modelled on $X_{4;3}^0,$ then $T$ meets a one-tetrahedron solid torus, $\lst_1,$ in a face. This is not possible due to Lemma \ref{lem:layered solid tori don't meet in single face}.

If the neighbourhood is modelled on $X_{4;3}^1,$ then either one or two tetrahedra of $X_{4;3}^1$ are mapped to $T.$ If one is mapped to $T,$ then $T$ cannot be maximal. If two are mapped to $T,$ then either $T$ is not maximal, or, as in the proof of Proposition~\ref{pro:degree three edges}, one observes that since the triangulation is minimal with at least two tetrahedra and $M$ is irreducible, there is a face pairing between the remaining two free faces of the third tetrahedron which implies that $M$ is a lens space with layered triangulation.

The remaining possibility is $X_{4;4},$ which implies that $e$ is contained in four distinct tetrahedra. This forces $e$ to be the univalent edge for $T.$
\end{proof}

Recall the notion of colouring from Subsection \ref{subsec:colouring}.

\begin{lemma}
Assume that the triangulation contains a single vertex and that all edge loops are coloured using a homomorphism $\varphi \co \pi_1(M) \to \Z_2.$ Then all tetrahedra in a layered solid torus in $M$ are either of type one or type three, but not both.
\end{lemma}
\begin{proof}
The colouring of the layered solid torus is uniquely determined by the image of the longitude under $\varphi;$ the result follows from the description of the layering procedure.
\end{proof}

\begin{definition}[(Types of layered solid tori)]
Assume that the triangulation contains a single vertex and that all edge loops are coloured using a homomorphism $\varphi \co \pi_1(M) \to \Z_2.$ A layered solid torus containing a tetrahedron of type one (respectively three) is accordingly termed of type one (respectively three). 
\end{definition}


\subsection{Maximal layered solid tori in atoroidal manifolds}

\begin{lemma}\label{lem:normal torus in 0-eff bounds solid torus}
Assume that $\tri$ is minimal and 0--efficient and that $M$ is atoroidal. Then every torus which is normal with respect to $\tri$ bounds a solid torus in $M$ on at least one side.
\end{lemma}

\begin{proof}
Let $T$ be a normal torus in $M.$ Since $M$ is atoroidal, there is a compression disc for $T.$ Denote the 2--sphere resulting from the compression by $S.$ Since $M$ is irreducible, $S$ bounds a ball, $B.$ If $T$ is not contained in $B,$ then a solid torus with boundary $T$ is obtained by attaching a handle to $B$ in $M.$ 

Hence assume that $T$ is contained in $B.$ Then $S$ can be pushed off $T,$ and $T$ is a barrier surface for $S.$ Since $\tri$ is 0--efficient, $S$ either shrinks to a 2--sphere embedded in a tetrahedron or to a vertex linking 2--sphere. It follows that $M = S^3.$ Alexander's torus theorem now implies that $T$ bounds a solid torus on at least one side.
\end{proof}

\begin{lemma}\label{lem:two meet in edge gives solid torus}
Assume that $\tri$ is minimal and 0--efficient and that $M$ is atoroidal. Suppose two maximal layered solid tori meet in precisely one edge. Then the boundary of a small regular neighbourhood of their union bounds a solid torus in M.
\end{lemma}

\begin{proof}
Denote the two maximal layered solid tori by $T_1, T_2,$ and the common edge by $e.$ Then the boundary of a small regular neighborhood $N$ of $T_1\cup T_2$ is a (topological) torus and a barrier surface. Hence, either $\partial N$ is isotopic to a normal surface or $\overline{M\setminus N}$ is a solid torus. In the second case, we are done, and in the first case Lemma \ref{lem:normal torus in 0-eff bounds solid torus} gives the conclusion.
\end{proof}

\begin{lemma}\label{lem:two meet in edge}
Assume that $\tri$ is minimal and 0--efficient and that $M$ is atoroidal. Suppose the edges are coloured by a homomorphism $\varphi \co \pi_1(M) \to \Z_2,$ and that two pairwise distinct maximal layered solid tori of type one meet in an $\varphi$--even edge. Then either $M$ is a lens space with layered triangulation, or $M$ admits a Seifert fibration with at least two and at most three exceptional fibres.
\end{lemma}

\begin{proof}
Denote the two maximal layered solid tori by $T_1, T_2,$ and the common $\varphi$--even edge by $e.$ If $T_1 \cap T_2$ properly contains $e,$ then $M$ is a lens space with layered triangulation according to Lemma \ref{lem:intersection of maximal layered solid tori}.

Hence assume $T_1 \cap T_2 = \{ e\}.$ Let $N$ be a small regular neighbourhood of $T_1 \cup T_2.$ Then $\partial N$ is an embedded torus in $M,$ and either $N$ or $\overline{M\setminus N}$ is a solid torus according to Lemma \ref{lem:two meet in edge gives solid torus}. Since $T_1$ and $T_2$ are of type one, $e$ is not a longitude of either of them. It follows that $N$ cannot be a solid torus. Whence $\overline{M\setminus N}$ is a solid torus. If $e$ is not homotopic to a longitude of $\overline{M\setminus N},$ then we have a Seifert fibration of $M$ with three exceptional fibres; otherwise we have a Seifert fibration with two exceptional fibres.
\end{proof}

\begin{lemma}\label{lem:three meet in edge}
Assume that $\tri$ is minimal and 0--efficient and that $M$ is atoroidal. Suppose the edges are coloured by a homomorphism $\varphi \co \pi_1(M) \to \Z_2,$ and that three pairwise distinct maximal layered solid tori meet in an $\varphi$--even edge. Then at least one of them is of type three unless $M$ admits a Seifert fibration with three exceptional fibres.
\end{lemma}

\begin{proof}
Denote the three maximal layered solid tori by $T_1, T_2, T_3,$ and the common $\varphi$--even edge by $e.$ If $T_i \cap T_j$ strictly contains $e$ for $i\neq j,$ then $M$ is a lens space with layered triangulation according to Lemma \ref{lem:intersection of maximal layered solid tori}. But then there are no three pairwise distinct maximal layered solid tori. Hence $M$ is not a lens space with layered triangulation.

Assume that $T_1$ and $T_2$ are of type one. As in the proof of Lemma~\ref{lem:two meet in edge}, let $N$ be a small regular neighbourhood of $T_1 \cup T_2.$ Then $\overline{M\setminus N}$ is a solid torus, but $N$ is not a solid torus as $e$ is not a longitude of $T_1$ or $T_2.$ Assume that $M$ does not admit a Seifert fibration with three exceptional fibres. Then $e$ is homotopic to a longitude of $\overline{M\setminus N}$ and it follows that $\varphi$ restricted to $\overline{M\setminus N}$ is trivial. If $T_3$ has $e$ as its longitude, then it is of type three. Otherwise, the longitude of $T_3$ is homotopic into $\overline{M\setminus N}.$ But this implies that $\varphi$ restricted to $T_3$ is trivial which again implies that $T_3$ is of type three.
\end{proof}

\begin{lemma}\label{lem:type one meet in even}
Assume that $\tri$ is minimal and 0--efficient and that $M$ is a lens space. Suppose the edges are coloured by a homomorphism $\varphi \co \pi_1(M) \to \Z_2.$

If two distinct type one maximal layered solid tori meet in an $\varphi$--even edge, then these are the only maximal layered solid tori of type one in the triangulation.
\end{lemma}

\begin{proof}
Denote the maximal layered solid tori by $T_1, T_2,$ and the common $\varphi$--even edge by $e.$ If $T_1 \cap T_2$ strictly contains $e,$ then $M$ is a lens space with layered triangulation, and hence $T_1$ and $T_2$ are the only maximal layered solid tori of type one in the triangulation. Hence assume $T_1 \cap T_2 = \{ e\}.$

Let $N$ be a small regular neighborhood of $T_1\cup T_2.$ As in the argument above, the closure of $\overline{M\setminus N}$ is a solid torus but $N$ is not a solid torus. However, since $M$ is assumed to be a lens space, $e$ must be a longitude of $\overline{M\setminus N}.$ Since $e$ is $\varphi$--even, $\varphi$ restricted to $e$ is trivial and thus $\varphi$ restricted to $\overline{M\setminus N}$ is trivial. Now assume that $T_3$ is a maximal layered solid torus of type one which is distinct from $T_1, T_2.$ Since $T_3$ is of type one, its longitude is $\varphi$--odd. But it is clearly homotopic into $\overline{M\setminus N},$ contradicting the fact that $\varphi$ restricted to $\overline{M\setminus N}$ is trivial. Whence $T_1$ and $T_2$ are the only maximal layered solid tori of type one in the triangulation.
\end{proof}


\section{Main result, consequences and examples}
\label{sec:minimal of L(2n,1)}


\subsection{The main result and some consequences}

\bf{Theorem \ref{thm:really main}}\rm\ \emph{
A lens space with even fundamental group satisfies Conjecture \ref{conj:layered lens is minimal} if it has a minimal layered triangulation such that there are no more $\varphi$--odd edges than there are $\varphi$--even edges. 
}

\begin{proof}
Let $M$ be a lens space with even fundamental group and $\varphi\co \pi_1(M)\to \Z_2$ be the unique non-trivial homomorphism. Suppose that the minimal layered triangulation has no more $\varphi$--odd edges than $\varphi$--even edges.

If $M= L(4,1),$ then $\lay_1$ is a one-tetrahedron triangulation satisfying the hypothesis and we only need to prove the uniqueness statement. This follows by enumerating all single tetrahedron triangulations of closed, orientable 3--manifolds.

If $M= \R P^3,$ then the minimal layered triangulation has two tetrahedra, one $\varphi$--odd edge and one $\varphi$--even edge. It is well-known that this is a minimal triangulation (as follows, for instance, from the above method of enumeration), but that there is another minimal triangulation with two vertices coming from the standard lens space representation. 

Hence we may assume that $M\neq \R P^3$ or $L(4,1).$ In this case, it is shown in Subsection~\ref{subsec:comb bds} that we have the following inequality for $\varphi$--even edges:
\begin{equation}\label{eqn:degree three proof}
\even_3 \ge 2 + \sum_{j=5}^{\infty} (j-4) \even_j.
\end{equation}
By way of contradiction, suppose that the given minimal triangulation is not the minimal layered triangulation. Each edge of degree three is $\varphi$--even, so (\ref{eqn:degree three proof}) implies that there are at least two edges of degree three. Since the triangulation contains at least three tetrahedra, each edge of degree three, $e,$ is the base edge of a layered solid torus subcomplex isomorphic to $\lst_2.$ This subcomplex is contained in a unique maximal layered solid torus, $\T(e).$ Conversely, if a maximal layered solid torus, $T,$ contains an edge, $e,$ of degree three, then $e$ is unique and we write $e=\e(T).$

Since we assume that the given minimal triangulation of $M$ is not the minimal layered triangulation, it follows from Lemma~\ref{lem:intersection of maximal layered solid tori} that any two distinct maximal layered solid tori share at most an edge. We seek a contradiction guided by inequality (\ref{eqn:degree three proof}).

The proof, a basic counting argument, is organised as follows. The set of all edges of degree three, $Y,$ is divided into a number of pairwise disjoint subsets. These subsets are bijectively pared with a set of pairwise distinct $\varphi$--even edges. If a subset of $Y$ contains $h$ pairwise distinct edges of degree three and the associated $\varphi$--even edge has degree $d,$ then we say that there is an associated \emph{deficit} of $h-d+4$ if $h \le d-4,$ and a \emph{gain} of $h-d+4$ if $h > d-4.$ Then (\ref{eqn:degree three proof}) implies that the total gain is at least two.

Let $e \in Y$ such that $\T(e)$ is of type three. Then $\lst_2 \cong T_0 \subseteq \T(e).$ Denote $e_0$ the longitude of $T_0.$ This is a $\varphi$--even boundary edge with $T_0$--degree $5$ and $M$--degree $5+m$ for some $m\ge1.$ The total number of maximal layered solid tori in $M$ meeting in $e_0$ is bounded above by $\frac{m+1}{2},$ since no two meet in a face. Hence the maximal layered solid tori containing a degree three edge and meeting in the $\varphi$--even edge $e_0$ contribute at most $\frac{m+1}{2}$ to the left hand side of (\ref{eqn:degree three proof}). The contribution of $e_0$ to the right hand side is $d(e_0)-4=1+m.$ One therefore obtains a \emph{deficit} since $\frac{m+1}{2}\le m+1.$ To $e_0$ associate the set, $Y_0,$ of all degree three edges $e'$ such that $\T(e')$ contains $e_0.$

We now proceed inductively. Let $e$ be an edge of degree three such that $\T(e)$ is of type three and $e$ is not contained in the collection of subsets $Y_0, ..., Y_{i-1}$ of $Y.$ Then $\T(e)$ contains a subcomplex isomorphic to $\lst_{2},$ whose longitude, $e_i,$ cannot be any of the edges $e_0,...,e_{i-1}.$ Consider the set $Y_i$ of all degree three edges $e'$ such that $\T(e')$ contains $e_i$ and $e'$ is not contained in any of $Y_0, ..., Y_{i-1}.$ Then the above calculation shows that there is a deficit associated to $e_i$ and $Y_i.$

It follows that there must also be a maximal layered solid torus of type one, $T,$ which contains an edge of degree three not in $\cup Y_i.$ Let $e$ be the unique $\varphi$--even boundary edge of $T.$ We observe:
\begin{enumerate}
\item Suppose $T_1$ is a maximal layered solid torus of type one which meets $T$ in $e.$ Lemma \ref{lem:type one meet in even} implies that $T, T_1$ are the only maximal layered solid tori of type one. If $d(e)\ge 5,$ then $e$ can be associated with $\{\e(T), \e(T_1)\}$ and either gives a deficit or a gain of $+1.$ In either case, (\ref{eqn:degree three proof}) cannot be satisfied since $Y = \cup Y_i \bigcup \{\e(T), \e(T_1)\}.$ Hence $d(e)= 4.$
\item If $T$ contains a $\varphi$--even interior edge, $e_T,$ of degree at least five, then $e_T$ can be associated with $\{\e(T)\}$ and gives a deficit.
\item If each $\varphi$--even interior edge of $T$ has degree less than five, but $d(e)\ge5,$ then it can be associated with $\{\e(T)\}$ and gives a deficit.
\end{enumerate}

Let $X$ be the set of all degree three edges such that there is an associated maximal layered solid torus of type one which satisfies (2) or (3) above. To the collection $\cup_j Y_j \bigcup X$ we have an associated deficit, and that its complement, $Z = Y \setminus (\cup_j Y_j \bigcup X),$ contains precisely the degree three edges with the property that the associated maximal layered solid torus has all $\varphi$--even interior edges of degree at most four and has a $\varphi$--even boundary edge of degree four. If $Z$ is empty, then (\ref{eqn:degree three proof}) is not satisfied.

\begin{figure}[t]
\psfrag{a}{{\small $-1$}}
\psfrag{b}{{\small $0$}}
\psfrag{c}{{\small $+1$}}
\begin{center}
      \includegraphics[height=4cm]{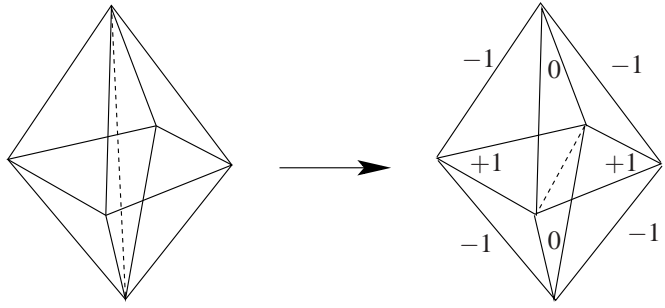}
\end{center}
    \caption{Edge flip: Labels indicate the changes in edge degree}
     \label{fig:edge flip}
\end{figure}
Let $T$ be the maximal layered solid torus associated to an element of $Z.$ It follows from Lemma~\ref{lem:layered lens char 2} that its $\varphi$--even boundary edge of degree four, $e,$ is its univalent edge, and that it is contained in four distinct tetrahedra in the triangulation. The argument proceeds by replacing the four tetrahedra around $e$ by a different constellation of four tetrahedra using an appropriate \emph{edge flip}, see Figure \ref{fig:edge flip}. The resulting triangulation $\tri'$ is also minimal and not a layered triangulation since it contains an edge of degree four contained in four distinct tetrahedra; there are associated sets $Y',$ $Y'_i,$ $X'$ and $Z'.$ We will show that either $Z'$ is a proper subset of $Z,$ or that $\tri'$ contains fewer tetrahedra of type three than $\tri.$ It then follows inductively that there is a minimal triangulation $\tri''$ which is not the minimal layered triangulation and with $Z'' = \emptyset;$ thus giving a contradiction.

It remains to describe the re-triangulation process. There are three different cases to consider; they are listed below using the types of tetrahedra ordered cyclically around the $e,$ starting with $T.$ The complex formed by the four tetrahedra around $e$ is also termed an octahedron (even though it may not be embedded), and the set of all four edges ``linking" $e$ its equator.

(1,1,1,1): Denote the maximal layered solid tori containing $e$ by $T$ and $T_1.$ (We may have $T_1=T.$) Doing any edge flip replaces the octahedron with four tetrahedra of type two meeting in an $\varphi$--even edge of degree four. The $\varphi$--even edges along the equator have their degrees increased by one. Since no maximal layered solid torus contains a tetrahedron of type two, it follows that each maximal layered solid torus with respect to $\tri$ is an maximal layered solid torus with respect to $\tri'$ except for $T$ and $T_1;$ whence $Y \setminus \{ \e(T), \e(T_1) \} \subset Y'.$ If $T$ is isomorphic to $\lst_2,$ then after the flip, the degree three edge becomes a degree four edge; so $\e(T) \notin Y',$ which implies $\e(T) \notin Z'.$ Otherwise $\e(T) \in Y',$ but $\T(\e(T))$ with respect to $\tri'$ has an $\varphi$--even boundary edge of degree five. Whence $\e(T) \in X',$ and so $\e(T) \notin Z'.$ Similarly for $T_1.$ Since no $\varphi$--even boundary edge has its degree decreased by the flip, and no other maximal layered solid tori are affected, we have $Y'_j = Y_j$ for each $j$ and $X \subseteq X'.$ This proves that $Z'$ is a proper subset of $Z$ in this case.

(1,2,2,1) or (1,1,2,2): We have that $T$ is the unique maximal layered solid torus containing $e.$ Do the unique edge flip such that the two $\varphi$--even edges not contained on the equator stay of the same degree. This replaces the octahedron by four tetrahedra of type one meeting in an odd edge of degree four. The main line of the argument is as above, showing that $\e(T) \notin Z'.$ If every other maximal layered solid torus with respect to $\tri$ is a maximal layered solid torus with respect to $\tri',$ then we are done. Hence assume that some maximal layered solid torus with respect to $\tri',$ $T'_0,$ is obtained from a maximal layered solid torus with respect to $\tri,$ $T_0,$ by layering onto it one of the (new) four tetrahedra of type one. Note that two of these tetrahedra cannot be contained in $T'_0$ for otherwise $T'_0$ meets another layered solid torus in a single face. Hence there is at most one such tetrahedron, and at most one such maximal layered solid torus which has been extended. Now also note that the $\varphi$--even boundary edge of $T'_0$ has degree one more than the $\varphi$--even boundary edge of $T_0.$ Thus, if $\e(T_0) \in X,$ then $\e(T'_0) \in X',$ and if $\e(T_0) \in Z,$ then also $\e(T'_0) \in X'.$ This proves that we again have $Y'_j = Y_j$ for each $j,$ $X \subseteq X',$ and $Z'$ is a proper subset of $Z.$

(1,2,3,2): In this case, doing any edge flip replaces this with a (1,1,2,2)--octahedron around an odd edge of degree four, thus reducing the number of tetrahedra of type three.
\end{proof}

\begin{corollary}
Let $L(2n,q)$ be a lens space with even fundamental group, and minimal layered triangulation $\lay_{2n,q}.$ Suppose that the edges have been coloured using the non-trivial homomorphism $\varphi \co \pi_1 (L(2n,q)) \to \Z_2.$ Then $\even(\lay_{2n,q}) \le \odd(\lay_{2n,q}).$
\end{corollary}

\begin{proof}
Assume that $\even(\lay_{2n,q}) > \odd(\lay_{2n,q}).$ Then Theorem \ref{thm:really main} implies that $\lay_{2n,q}$ is the unique minimal triangulation. It follows from the proof that $L(2n,q) \neq \R P^3$ or $L(4,1),$ since then $\even(\lay_{2n,q}) = \odd(\lay_{2n,q}).$ Hence (\ref{inequ:odd and even}) implies
\begin{equation*}
4 - 2T(\lay_{2n,q}) - 4 \chi(S_\varphi(\lay_{2n,q})) \ge 2 + 2 \even(\lay_{2n,q}) - 2 \odd(\lay_{2n,q}) > 2.
\end{equation*}
Using (\ref{inequ for min tri general}), this gives $\even_3 >2,$ which is not possible in a layered triangulation.
\end{proof}

\begin{corollary}\label{cor:example construction}
Let $L(2n,q)$ be a lens space with even fundamental group, and not homeomorphic to $\R P^3$ or $L(4,1).$ Denote the minimal layered triangulation $\lay_{2n,q},$ and suppose that the edges have been coloured using the non-trivial homomorphism $\varphi \co \pi_1 (L(2n,q)) \to \Z_2.$ The following are equivalent:
\begin{enumerate}
\item $\even(\lay_{2n,q}) \ge \odd(\lay_{2n,q});$
\item $\even(\lay_{2n,q}) = \odd(\lay_{2n,q});$
\item there are precisely two $\varphi$--even edges of degree three, and all other $\varphi$--even edges are of degree four.
\end{enumerate}
Moreover, the only layered triangulation of $\lay_{2n,q}$ satisfying (3) is the minimal layered triangulation.
\end{corollary}

\begin{proof}
We have $(1) \Leftrightarrow (2)$ due to the preceding corollary. We have $(1) \Rightarrow (3)$ since (\ref{eqn:degree three proof}) holds with $\even_3=2.$ 

Hence assume that we are given an arbitrary layered triangulation $\tri$ of $L(2n,q)$ with precisely two $\varphi$--even edges of degree three, and all other $\varphi$--even edges are of degree four. We will show that this implies that there are as many $\varphi$--even as $\varphi$--odd edges.

Each tetrahedron in $\tri$ is of type one and we may describe $\tri$ starting from $T_1 \cong \lst_1=\{1,2,3\}$ with $e_1$ and $e_3$ $\varphi$--odd, and $e_2$ $\varphi$--even. Since $\tri$ contains two edges of degree three, there is a tetrahedron layered on $e_2;$ the resulting layered solid torus $T_2\cong \lst_2=\{1,3,4\}$ has univalent edge the $\varphi$--even edge, $e_4.$ Hence in $\tri,$ there is a tetrahedron layered on either of the $\varphi$--odd boundary edges $e_1$ or $e_3$ of $T_2,$ giving a layered solid torus $T_3$ which is either isomorphic to $\{3,4,7\}$ or $\{1,4,5\}$ with $\varphi$--odd univalent edge. The even boundary edge $e_4$ has $T_3$--degree three; hence either $\tri$ is obtained from folding along $e_4$ or by layering on $e_4.$ The construction continues inductively. Since we alternate between layering on $\varphi$--even and $\varphi$--odd edges, and in the end fold along an $\varphi$--even edge of degree three, it follows that the resulting triangulation has as many $\varphi$--odd edges as $\varphi$--even edges. This, in particular, proves $(3) \Rightarrow (2)$ since the minimal layered triangulation must be constructed in this way.

To prove the last statement, notice that the above procedure describes the sub-tree of the $L$--graph of \cite{JR:LT} shown in Figure \ref{fig:L-graph}, and that any layered triangulation satisfying (3) corresponds to a unique path without looping or backtracking in this sub-tree. Hence it is a minimal layered triangulation.
\end{proof}

\begin{remark}
It is shown in \cite{JR:LT} that the minimal layered extension of $\{p,q,p+q\}$ contains $E(q,p)-1$ tetrahedra, where $q>p>0.$ This implies that upon folding along edge $q,$ the minimal layered triangulation of $L(2p+q, p)$ contains $E(q,p)-1$ tetrahedra. The formula for the complexity given in the introduction follows from $E(2p+q, p) = E(q,p) +2.$
\end{remark}

\begin{figure}[t]
\psfrag{aa}{{\scriptsize $0/1$}}
\psfrag{bb}{{\scriptsize $1/1$}}
\psfrag{cc}{{\scriptsize $1/2$}}
\psfrag{dd}{{\scriptsize $1/3$}}
\psfrag{ee}{{\scriptsize $3/4$}}
\psfrag{ff}{{\scriptsize $1/4$}}
\psfrag{gg}{{\scriptsize $4/7$}}
\psfrag{hh}{{\scriptsize $3/7$}}
\psfrag{ii}{{\scriptsize $4/5$}}
\psfrag{jj}{{\scriptsize $1/5$}}
\psfrag{kk}{{\scriptsize $7/10$}}
\psfrag{ll}{{\scriptsize $3/10$}}
\psfrag{mm}{{\scriptsize $5/6$}}
\psfrag{nn}{{\scriptsize $1/6$}}
\psfrag{oo}{{\scriptsize $2/3$}}
\psfrag{pp}{{\scriptsize $3/5$}}
\psfrag{qq}{{\scriptsize $2/5$}}
\psfrag{rr}{{\scriptsize $7/17$}}
\psfrag{ss}{{\scriptsize $3/13$}}
\psfrag{tt}{{\scriptsize $5/11$}}
\psfrag{uu}{{\scriptsize $1/7$}}
\psfrag{vv}{{\scriptsize $17/24$}}
\psfrag{ww}{{\scriptsize $7/24$}}
\psfrag{xx}{{\scriptsize $13/16$}}
\psfrag{yy}{{\scriptsize $3/16$}}
\psfrag{zz}{{\scriptsize $11/16$}}
\psfrag{za}{{\scriptsize $5/16$}}
\psfrag{zb}{{\scriptsize $7/8$}}
\psfrag{zc}{{\scriptsize $1/8$}}
\psfrag{ab}{{\scriptsize $p/q$}}
\psfrag{bc}{{\scriptsize $q/p+q$}}
\psfrag{cd}{{\scriptsize $p/p+q$}}
\begin{center}
      \includegraphics[width=12cm]{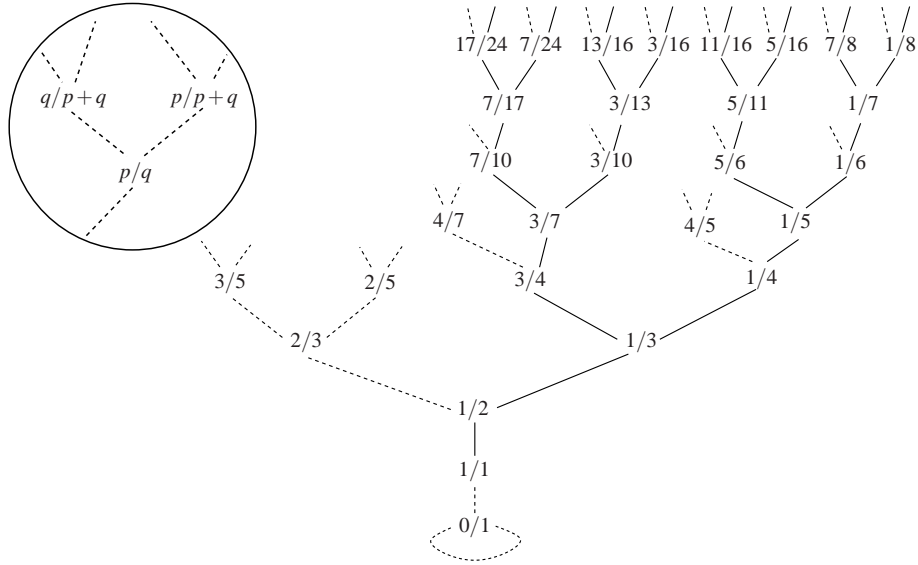}
\end{center}
\caption{A complete description of the set of all lens spaces to which Theorem \ref{thm:really main} applies: Shown is part of the $L$--graph of \cite{JR:LT}. The 0--cells are in bijective correspondence with one--vertex triangulations on the boundary of the solid torus. The unique path from $p/q$ to $1/1$ without backtracking describes the minimal layered triangulation of the solid torus extending $\{p,q,p+q\}$ when $q>p>0.$ The minimal layered triangulation of a lens space satisfies the hypothesis of Theorem \ref{thm:really main} if and only if it is obtained by folding the minimal layered extension of $\{p,q,p+q\}$ along $q,$ where $p/q\neq 1/2$ is a vertex of valence two in the infinite subtree made up of the solid segments. The associated lens space is then $L(2p+q,p).$}
\label{fig:L-graph}
\end{figure}


\subsection{Families of examples}

As noted earlier, each of the lens spaces $L(2n,1)$ contains as many $\varphi$--odd edges as $\varphi$--even edges, and hence satisfies Conjecture \ref{conj:layered lens is minimal}. To recognise more families that satisfy this condition, we study layered solid tori that are isomorphic to $\lst_m$ for some $m\ge 2.$

\begin{definition}($\lst$--maximal layered solid torus)
A layered solid torus is an \emph{$\lst$--maximal layered solid torus with respect to $\tri$} in $M$ if it is isomorphic to $\lst_m$ for some $m\ge 2$ and it is not strictly contained in any other layered solid torus isomorphic to $\lst_h$ for some $h>m.$
\end{definition}

If $\tri$ has more than two tetrahedra, then any edge of degree three is contained in an $\lst$--maximal layered solid torus. Moreover, this $\lst$--maximal layered solid torus has the edge of degree three as its base-edge and all other interior edges have degree four.

To begin with, we analyse the result of identifying two layered solid tori, $T_1\cong \lst_t$ and $T_2\cong \lst_s,$ $t \ge s \ge 1,$ along their boundaries. The identification of the solid tori is uniquely determined by a pairing of the boundary edges, and hence there are six possibilities:

$(s+2,s+1,1) \leftrightarrow (t+1,t+2,1):$ This gives $L(s+t+3,1)$ with minimal layered triangulation (as can be seem from the continued fraction expansion), and if $s+t > 2,$ there are precisely two $\lst$--maximal layered solid tori which meet in $s+t-2$ tetrahedra. 

$(s+2,s+1,1) \leftrightarrow (t+1,1,t+2):$ This gives $L((s+1)(t+2)+1,t+2)$ with minimal layered triangulation. If $s \ge 2,$ there are precisely two $\lst$--maximal layered solid tori which meet in one tetrahedron; one is isomorphic to $\lst_s,$ the other to $\lst_{t+1}.$ If $s=1,$ then there is at most one $\lst$--maximal layered solid torus.

$(s+2,s+1,1) \leftrightarrow (1,t+2,t+1):$ This gives $L((s+2)(t+1)+1,t+1)$ with minimal layered triangulation. If $s \ge 2,$ there are precisely two $\lst$--maximal layered solid tori which meet in one tetrahedron; one is isomorphic to $\lst_{s+1},$ the other to $\lst_{t}.$ If $s=1,$ then there is at most one $\lst$--maximal layered solid torus.

$(s+2,s+1,1)  \leftrightarrow (t+2,t+1,1):$ This gives $L(t-s,1),$ and there is an edge of degree two. The triangulation is not a minimal layered triangulation (since neither $t-s=3$ and $t+s=2$ nor $t-s=4$ and $t+s=1$ has integral solutions).

$(s+2,s+1,1) \leftrightarrow (t+2,1,t+1):$ This gives $L(s(t+1)+t,t+1)$ and there is an edge of degree two. An algebraic argument shows again that the triangulation is not a minimal layered triangulation unless $s=t=1$ and $M=L(3,1).$ In particular, the triangulation does not contain an $\lst$--maximal layered solid torus.

$(s+2,s+1,1) \leftrightarrow (1,t+1,t+2):$ This gives $L((s+1)(t+2)+t+1,t+2)$ with a minimal layered triangulation. If $s \ge 2,$ there are precisely two $\lst$--maximal layered solid tori which meet in two faces.

\begin{lemma}\label{lem:min lens char 1}
If two distinct $\lst$--maximal layered solid tori meet in precisely two faces then $M$ is a lens space with layered triangulation. If it is the minimal layered triangulation, then $M$ is homeomorphic to $L((s+1)(t+2)+t+1,t+2),$ where $t \ge s \ge 2.$
\end{lemma}

\begin{proof}
It follows from Lemma \ref{lem:meet in 2 faces} that $M$ is a lens space with layered triangulation. An $\lst$--maximal layered solid torus contains at least two tetrahedra. The lemma now follows directly from the stated possibilities.
\end{proof}

Note that if $M$ as in the above lemma has even fundamental group, then $s+2 \leftrightarrow 1$ and $1 \leftrightarrow t+2$ imply that $s$ and $t$ are odd; hence the number of tetrahedra is even. In particular, the number of edges in the triangulation is odd, so the number of $\varphi$--odd edges is greater than the number of $\varphi$--even edges.

Recall that $L(p_1, q_1) = L(p_2, q_2)$ if and only if $p_1=p_2$ and $q_1q_2^{\pm 1} \equiv \pm 1(p_1).$ The following two results give Theorems~\ref{thm:main}, \ref{thm:main2} and \ref{thm:main3}.

\begin{proposition}\label{pro:layered lens char 1}
If two distinct $\lst$--maximal layered solid tori share a tetrahedron, then $M$ is a lens space with minimal layered triangulation. Moreover, $M$ is contained in one of the following, mutually exclusive families:
\begin{enumerate}
\item $L(s+t+3,1),$ $t\ge s\ge 1;$
\item $L( (s+2)(t+1)+1 , t+1),$ where $t >  s >1,$
\item $L( (s+1)(t+2)+1 , t+2),$ where $t > s > 1,$
\item $L( (t+1)(t+2)+1 , t+2)=L( (t+1)(t+2)+1 , t+1)$ where $t \ge 2.$
\end{enumerate}
\end{proposition}

\begin{proof}
Since the $\lst$--maximal layered solid tori are distinct, it follows from Lemma~\ref{lem:meet in tet} that $M$ is a lens space. Moreover, $M$ can be described as the union of two layered solid tori, $T_1$ and $T_2,$ along their boundary, where $T_1 \cong \lst_s,$ and $T_2\cong \lst_t,$ with $s,t \ge 1.$ Without loss of generality, assume $t \ge s \ge 1.$ The above possibilities show that, under the assumption that there are two distinct $\lst$--maximal layered solid tori, $M$ is one of the cases listed in (1)--(4). The equality in (4) follows since $(t+1)(t+2) \equiv -1$ modulo $(t+1)(t+2)+1.$ Also note that a lens space as in (4) with $t=1$ has a minimal layered triangulation with two tetrahedra and does not contain an $\lst$--maximal layered solid torus. It remains to show that the cases are mutually exclusive.

Since the minimal layered triangulation is unique, it follows from the description of intersections of $\lst$--maximal layered solid tori that the lens spaces in the first list do not appear in the second or third unless $s+t-2=1,$ whence $M=L(6,1).$ This does not appear in (2) or (3). It remains to show that the possibilities in (2) and (3) are complete and exclusive.

Let $t_1\ge s_1 \ge 1$ and $t_2\ge s_2 \ge 1.$ The minimal layered triangulations of $L( (s_1+2)(t_1+1)+1 , t_1+1)$ and $L( (s_2+1)(t_2+2)+1 , t_2+2)$ are characterised by the fact that there are two $\lst$--ma-la-so-tos meeting in one tetrahedron. It follows that ($s_1=s_2+1$ and $t_1+1=t_2$) or ($s_1=t_2$ and $t_1+1=s_2+1$). In the second case $s_1 = t_2 \ge s_2 = t_1 \ge s_1,$ and hence these are the cases not contained in (2) or (3), but in (4). It follows that there are elements in (2) and (3) such that
$$L( (s_1+2)(t_1+1)+1 , t_1+1)= L( (s_2+1)(t_2+2)+1 , t_2+2)$$
only if $s_1=s_2+1$ and $t_1+1=t_2.$ Since the fundamental groups must have the same order, we have $t_2=s_2+1,$ giving $t_1+1=t_2=s_2+1=s_1\le t_1;$ a contradiction.
\end{proof}

\begin{corollary}
If $M$ is one of the lens spaces listed in (1)--(4) of Proposition~\ref{pro:layered lens char 1} and has even fundamental group, then there are as many $\varphi$--even edges as there are $\varphi$--odd edges, where $\varphi\co \pi_1(M) \to Z_2$ is the unique non-trivial homomorphism. Moreover, $M$ is either in 
\begin{enumerate}
\item[] (1) and $s+t$ is odd;
\item[] (2) and $s$ is odd and $t$ is even;
\item[] (3) and $s$ is even and $t$ is odd.
\end{enumerate}
\end{corollary}
\begin{proof}
First note that every lens space in (4) has odd fundamental group. We give the argument for the third case; the others are analogous. Since the gluing is $(s+2,s+1,1) \leftrightarrow (t+1,1,t+2)$ to give $L((s+1)(t+2)+1,t+2),$ and we want to have even fundamental group, the paring $1 \leftrightarrow t+2$ forces $t$ odd, and $s+1\leftrightarrow 1$ forces $s$ even. Even edges in the $\lst$--layered solid tori correspond to $\varphi$--even edges in the lens space; similarly for odd edges. Since the gluing identifies one pair of even edges and two pairs of odd edges, we have
$$
\frac{1}{2}(s+2) + \frac{1}{2}(t+1) -1 = \frac{1}{2}(s+t+1)
$$
$\varphi$--even edges in $L((s+1)(t+2)+1,t+2),$ and
$$
\frac{1}{2}(s+2) + \frac{1}{2}(t+3) -2 = \frac{1}{2}(s+t+1)
$$
$\varphi$--odd edges in $L((s+1)(t+2)+1,t+2).$
\end{proof}

Any other example to which Theorem \ref{thm:really main} applies contains two $\lst$--maximal layered solid tori which meet in precisely two edges, one edge or not at all. Every example can be constructed as in the proof of Corollary \ref{cor:example construction}. For instance, take $s\ge 2$ even, start with $\lst_s = \{1, s+1, s+2\},$ then in turn layer on $1,$ $s+2,$ $s+1,$ $3s+4,$ $2s+3,$ $7s+10,$ $5s+7$ and fold along $17s+24.$ This yields $L(41s+58, 12s+17)$ with minimal layered triangulation having $s+7$ tetrahedra and as many even edges as odd edges. The continued fractions expansion verifies the number of tetrahedra:
$$\frac{41s+58}{12s+17} = [3,2,2,2,s+1],$$
and hence $E(41s+58, 12s+17) - 3 = (s + 10)-3 = s+7.$




\address{Department of Mathematics, Oklahoma State University, Stillwater, OK 74078-1058, USA}
\email{jaco@math.okstate.edu}

\address{Department of Mathematics and Statistics, The University of Melbourne, VIC 3010, Australia} 
\email{rubin@ms.unimelb.edu.au} 

\address{Department of Mathematics and Statistics, The University of Melbourne, VIC 3010, Australia} 
\email{tillmann@ms.unimelb.edu.au} 
\Addresses
                                                      
\end{document}